\documentclass[12pt]{amsart}

\usepackage[T1]{fontenc}
\usepackage[utf8]{inputenc}
\usepackage{lmodern}
\usepackage[a4paper,margin=1in]{geometry}
\usepackage{amsmath,amssymb,amsthm,mathtools}
\usepackage{enumitem}
\usepackage{tikz}
\usepackage{float}
\usepackage{graphicx}
\usepackage[colorlinks=true,linkcolor=blue,citecolor=blue,urlcolor=blue]{hyperref}
\usepackage{microtype}
\hypersetup{
 pdfauthor={Rongli Huang and Qinfeng Jiang},
 pdftitle={Sharp Boundary Recession Criteria for the Special Lagrangian Curvature Potential Equation},
 pdfsubject={Boundary second derivative estimates and sharp recession criteria},
 pdfkeywords={special Lagrangian curvature potential equation, mixed recession compatibility, boundary Hessian estimate}
}
\allowdisplaybreaks

\numberwithin{equation}{section}

\newtheorem{theorem}{Theorem}[section]
\newtheorem{proposition}[theorem]{Proposition}
\newtheorem{lemma}[theorem]{Lemma}
\newtheorem{corollary}[theorem]{Corollary}
\theoremstyle{definition}
\newtheorem{definition}[theorem]{Definition}
\newtheorem{example}[theorem]{Example}
\theoremstyle{remark}
\newtheorem{remark}[theorem]{Remark}
\numberwithin{figure}{section}

\newcommand{\R}{\mathbb R}
\newcommand{\Sym}{\operatorname{Sym}}
\newcommand{\tr}{\operatorname{tr}}
\newcommand{\diag}{\operatorname{diag}}
\newcommand{\cJ}{\mathcal J}
\newcommand{\cL}{\mathcal L}
\newcommand{\cR}{\mathcal R}
\newcommand{\Om}{\Omega}
\newcommand{\eps}{\varepsilon}
\newcommand{\II}{\mathrm{II}}
\newcommand{\ip}[2]{\left\langle #1,#2\right\rangle}

\title[Sharp Boundary Recession Criteria]{Sharp Boundary Recession Criteria for the Special Lagrangian Curvature Potential Equation}

\author[R. Huang]{Rongli Huang}
\address{School of Mathematics and Statistics,
	Guangxi Normal University, Guilin, Guangxi, 541004, China}
\email{ronglihuangmath@gxnu.edu.cn}

\thanks{The first author was supported by Guangxi Natural Science Foundation (No.2026GXNSFDA00640012)}

\author[Q. Jiang]{Qinfeng Jiang}
\address{School of Mathematical Sciences, Beijing Normal University, Beijing 100875, China}
\email{202531130031@mail.bnu.edu.cn}


\begin{document}

\begin{abstract}
We establish boundary second derivative estimates for convex graphical solutions of the special Lagrangian curvature potential equation.  Since the curvature matrix depends on both $Du$ and $D^2u$, a phase subsolution alone does not provide the full linearized separation needed for the mixed derivative estimate.  We introduce a mixed recession compatibility condition imposed only on doubly degenerate level jets.  It yields a uniform mixed derivative bound and is sharp within the class of fixed smooth zero-order barriers considered here.  For the double-normal derivative, an exact complex Schur-complement identity gives
\[
 u_{\nu\nu}=\beta+\alpha\cot\delta,
 \qquad 1\leq\alpha\leq C,
 \qquad |\beta|\leq C,
\]
where $\alpha$ and $\beta$ are explicit Schur-complement coefficients,
$\delta$ is the actual boundary limiting-phase gap, and $C$ depends only on
uniform bounds for the gradient and the mixed boundary derivatives.  Thus
curvature blows up if and only if this gap collapses, with optimal rate
$\delta^{-1}$.  Smooth radial solutions attain the rate, while a rank-loss
model shows that a strict lower subsolution need not force strict convexity.
\end{abstract}

\maketitle

\section{Introduction and main results}\label{sec:intro}

Let $\Om\subset\R^n$, $n\geq2$, be a bounded smooth domain and let
\[
 \Gamma_u=\{(x,u(x)):x\in\Om\}\subset\R^{n+1}
\]
be the graph of a function $u$.  We denote by
$\kappa_1[u],\ldots,\kappa_n[u]$ the principal curvatures of $\Gamma_u$ with respect to the upward unit normal.  The constant phase special Lagrangian curvature potential equation is
\begin{equation}\label{eq:slc-intro}
 \sum_{i=1}^n\arctan\kappa_i[u]=\theta,
 \qquad 0<\theta<\frac{n\pi}{2}.
\end{equation}
The equation is the graph-curvature analogue of the classical special Lagrangian equation
\[
 \sum_{i=1}^n\arctan\lambda_i(D^2u)=\Theta.
\]
The latter originates in calibrated geometry \cite{HarveyLawson}, and the geometric notion of special Lagrangian curvature was developed by Smith \cite{Smith}.  The analytical distinction between the two equations is important here: the principal curvatures of a graph depend on both $D^2u$ and $Du$.

Let
\begin{equation*}
 w(p)=\sqrt{1+|p|^2},\qquad
 B(p)=(I+p\otimes p)^{-1/2},\qquad
 A[u]=\frac1{w(Du)}B(Du)D^2uB(Du).
\end{equation*}
The principal curvatures of $\Gamma_u$ are the eigenvalues of $A[u]$.  With
\[
 F(M)=\sum_{i=1}^n\arctan\lambda_i(M),
\]
equation \eqref{eq:slc-intro} becomes $F(A[u])=\theta$.  Throughout the paper we work on the convex branch $D^2u\geq0$.  On the positive cone $F$ is elliptic and concave in the curvature matrix $A$, but the composite map
\[
 (p,X)\longmapsto F\!\left(\frac1{w(p)}B(p)XB(p)\right)
\]
is not jointly concave.  Consequently, concavity in the Hessian variable alone does not yield the mixed boundary estimate for this gradient-dependent equation.

The Dirichlet theory for graph curvature equations goes back to the Caffarelli--Nirenberg--Spruck program \cite{CNSV}; related boundary and curvature estimates appear in \cite{Trudinger1990,Ivochkina,ShengUrbasWang,GuanRenWang}.  For Hessian equations, subsolutions and tangent cones at infinity yield nearly optimal second derivative hypotheses \cite{Guan1994,Guan2014,Szekelyhidi}.  Trudinger's finite asymptotic limit, obtained by sending the normal Hessian eigenvalue to $+\infty$ \cite{Trudinger1995}, is a precursor of the limiting-phase gap used below.  The Dirichlet problem for the Hessian Lagrangian phase equation was studied in \cite{Bhattacharya,CollinsPicardWu}.

The classical special Lagrangian equation provides the closest Hessian counterpart to \eqref{eq:slc-intro}.  Chen--Yuan--Warren \cite{ChenWarrenYuan} proved an interior Hessian estimate on the convex branch.  In dimension three, Warren--Yuan \cite{WarrenYuan2010} obtained Hessian and gradient estimates for large phase; they also established the gradient estimate in the critical and supercritical regimes in general dimensions.  Wang--Yuan \cite{WangYuan2014} subsequently proved the corresponding Hessian estimate in all dimensions.  At the global level, Yuan's Bernstein theorems \cite{Yuan2002,Yuan2006} show that convex entire solutions, and more generally entire solutions with supercritical phase, are quadratic.  These results identify phase and convexity as two basic sources of regularity.

The restrictions are essential.  In the subcritical regime, singular solutions were constructed by Nadirashvili--Vl\u{a}du\c{t} \cite{NadirashviliVladut} and Wang--Yuan \cite{WangYuan2013}; Mooney--Savin \cite{MooneySavin} later produced Lipschitz viscosity solutions which are not $C^1$.

The dependence of $A[u]$ on $Du$ creates a first-order drift in the linearization.  Such terms are treated in augmented-Hessian theory \cite{JiangTrudingerYang,JiangTrudinger} and in the graph-curvature boundary estimates of Jiao--Sun \cite{JiaoSun}.  Harvey--Lawson \cite{HarveyLawson2021} computed the asymptotic interior and the optimal viscosity pseudoconvexity for a special Lagrangian theory that includes the graph-curvature operator; Xiang--Xiong \cite{XiangXiong} studied a generalized curvature version in the critical and supercritical ranges.  The mixed recession condition used here couples subsolution separation with the drift-corrected boundary curvature only when both standard quadratic terms degenerate.

An analogous interior theory holds for the special Lagrangian curvature equation.  Qiu--Zhou \cite{QiuZhou} obtained interior Hessian estimates at the critical phase and for convex graphical solutions, as well as interior gradient estimates for arbitrary constant phase.  Qiu--Tao \cite{QiuTao} constructed Lipschitz non-$C^1$ viscosity solutions in dimensions two and three and a smooth two-dimensional nonconvex curvature blow-up family.  These results show that phase and convexity continue to govern regularity in the curvature setting, although the operator now depends on the gradient.  The examples in \cite{QiuTao} concern interior degeneration on nonconvex or subcritical branches.  We work instead on the convex branch and study two boundary phenomena: failure of the mixed derivative estimate and collapse of the limiting-phase gap as $u_{\nu\nu}\to+\infty$.  They require separate arguments.

Throughout the paper, $\nu$ is the inward unit normal and
$\II_{\partial\Om}$ is normalized so that a convex boundary has positive
second fundamental form.  The symbol $C$ denotes a positive constant whose
dependence is displayed when it matters; unless stated otherwise, it is
independent of the solution, the homotopy parameter $t$, and the recession
size.  All matrix norms are operator norms, and $X:Y=\tr(XY)$ for symmetric
 matrices, while $x\cdot y$ denotes the Euclidean inner product.  The subscript $T$ denotes tangential restriction to
$T\partial\Om$; $D_{\partial\Om}^2$ is the intrinsic boundary Hessian, and
$I_T$ is the identity on $T\partial\Om$.  We write $\Sym(n)$ for the space
of real symmetric $n\times n$ matrices.

We begin by placing the equation in the vertical graph homotopy.  For $t\in[0,1]$ set
\begin{equation}\label{eq:homotopy-intro}
 w_t(p)=\sqrt{1+t^2|p|^2},\qquad
 B_t(p)=(I+t^2p\otimes p)^{-1/2},
\end{equation}
and
\begin{equation*}
 A_t(p,X)=\frac1{w_t(p)}B_t(p)XB_t(p),
 \qquad G_t(p,X)=F(A_t(p,X)).
\end{equation*}
We abbreviate $A_t[u]=A_t(Du,D^2u)$.  Thus $A_0[u]=D^2u$ and $A_1[u]=A[u]$.

Let $K\subset\R^n$ be a fixed compact convex set containing the gradients of the solutions under consideration.  On a level jet
\begin{equation}\label{eq:level-jet-intro}
 (t,p,X)\in[0,1]\times K\times\Sym(n),
 \qquad X\geq0,
 \qquad G_t(p,X)=\theta,
\end{equation}
 write
\begin{equation*}
 a=(a^{ij})=G_{t,X}(p,X),\qquad
 b=(b^k)=G_{t,p}(p,X),
\end{equation*}
that is,
\begin{equation*}
 a^{ij}=\frac{\partial G_t}{\partial X_{ij}}(p,X),
 \qquad
 b^k=\frac{\partial G_t}{\partial p_k}(p,X).
\end{equation*}
Let $\underline u\in C^3(\overline\Om)$ be a fixed convex function with the same boundary value as $u$.  We define the full linearized separation
\begin{equation*}
 \cJ(x;t,p,X)
 =a:\bigl(D^2\underline u(x)-X\bigr)
  +b\cdot\bigl(D\underline u(x)-p\bigr).
\end{equation*}
Let $d$ be the inward boundary distance in a fixed collar of $\partial\Om$ and set
\begin{align*}
 \ell(x;t,p,X)&=a:D^2d(x)+b\cdot Dd(x),\\
 Q(x;t,p,X)&=(p-D\underline u)^Ta(p-D\underline u)=a^{ij}(p_i-\underline u_i)(p_j-\underline u_j)\geq0.
\end{align*}
On $\partial\Om$, the distance function satisfies
$D^2d(\nu,\cdot)=0$ and
$D_T^2d=-\II_{\partial\Om}$.  Hence, writing $a_T$ for the
restriction of $a$ to $T\partial\Om$, we have
\begin{equation}\label{eq:ell-boundary-intro}
	-\ell=a_T:\II_{\partial\Om}-b\cdot\nu.
\end{equation}
Here $a_T:\II_{\partial\Om}$ is the ellipticity-weighted boundary
curvature, while $-b\cdot\nu$ is the normal contribution of the
gradient drift.  Their sum is the effective boundary curvature seen
by the full linearized operator.

We first isolate the recession sequences on which both quadratic terms lose coercivity.

\begin{definition}\label{def:ddrs-intro}
Fix $x_0\in\partial\Om$.  A sequence
$(x_m,t_m,p_m,X_m)$ of level jets, with $x_m$ in the fixed boundary collar, is called a doubly degenerate recession sequence based at $x_0$ if
\begin{gather}\notag
 x_m\to x_0,\qquad |X_m|\to\infty,\\
 a^{(m)}(Dd(x_m),Dd(x_m))\to0,
 \qquad Q^{(m)}\to0.\label{eq:ddrs-b-intro}
\end{gather}
Here $a^{(m)}$, $b^{(m)}$, $\cJ^{(m)}$, $\ell^{(m)}$, and
$Q^{(m)}$ denote the corresponding quantities evaluated at
$(x_m,t_m,p_m,X_m)$ and we regard $a$ as a bilinear form and write
$a(\xi,\eta)=a^{ij}\xi_i\eta_j$.  A \emph{recession limit pair} is a subsequential limit
\begin{equation*}
 (j,g)=\lim_{m\to\infty}(\cJ^{(m)},-\ell^{(m)}).
\end{equation*}
We denote the set of all such pairs by $\cR^*_{x_0}$.  If no such sequence exists, we put $\cR^*_{x_0}=\varnothing$.
\end{definition}

Lemma \ref{lem:coefficient-bounds} shows that $\cJ$ and $\ell$ are uniformly bounded on level jets; a diagonal argument then shows that $\cR^*_{x_0}$ is compact.  The two limits in \eqref{eq:ddrs-b-intro} mean that the normal distance-square term and the square generated by concavifying $u-\underline u$ have simultaneously lost coercivity.

This leads to the compatibility condition used in the mixed estimate.

\begin{definition}\label{def:MRC-intro}
We say that $\underline u$ satisfies the mixed recession compatibility condition at $x_0\in\partial\Om$ if there is a number $\tau_{x_0}>0$ such that
\begin{equation}\label{eq:MRC-intro}
 \inf_{(j,g)\in\cR^*_{x_0}}(j+\tau_{x_0}g)>0,
\end{equation}
where the infimum over the empty set is $+\infty$.  We say that (MRC) holds on $\partial\Om$ if it holds at every boundary point.
\end{definition}

Equivalently, along every doubly degenerate recession sequence based at $x_0$,
\begin{equation}\label{eq:MRC-seq-intro}
 \liminf_{m\to\infty}
 \bigl(\cJ^{(m)}-\tau_{x_0}\ell^{(m)}\bigr)>0
\end{equation}
with a common positive margin.  Since $x_m\to x_0$ and the coefficients are bounded,
 \[
  -\ell^{(m)}-
  \bigl((a^{(m)})_{T_{x_0}}:\II_{\partial\Om}(x_0)
        -b^{(m)}\cdot\nu(x_0)\bigr)\longrightarrow0.
 \]
Thus (MRC) allows subsolution separation and the drift-corrected boundary curvature in \eqref{eq:ell-boundary-intro} to compensate one another precisely on the doubly degenerate recession set.
 
 \begin{remark}\label{rem:MRC-natural}
 	Condition (MRC) is imposed only on recession jets for which both
 	quadratic coercive terms in the standard boundary barrier vanish.   Appendix~\ref{app:MRC-examples} records
 	four elementary situations in which (MRC) follows from familiar
 	analytic or geometric hypotheses.  These examples also show that the
 	condition reduces to standard strict-subsolution or boundary-convexity
 	criteria when one of the two mechanisms is already coercive.
 \end{remark}

Our first result is the mixed derivative estimate.

\begin{theorem}\label{thm:mixed-intro}
 Let $\Om$ be a bounded $C^4$ domain, let $\underline u\in C^3(\overline\Om)$ be convex, and let $K\subset\R^n$ be compact and convex.  Suppose that $u\in C^4(\overline\Om)$ is a convex solution of
\begin{equation}\label{eq:Pt-intro}
 G_t(Du,D^2u)=\theta\quad\text{in }\Om,
 \qquad u=\underline u\quad\text{on }\partial\Om
\end{equation}
 for an arbitrary $t\in[0,1]$, and assume
\begin{equation*}
 Du(\overline\Om)\subset K,
 \qquad u\geq\underline u\quad\text{in }\Om.
\end{equation*}
If (MRC) holds on $\partial\Om$, then
\begin{equation}\label{eq:mixed-est-intro}
 \sup_{x\in\partial\Om}\sup_{\substack{\xi\in T_x\partial\Om\\|\xi|=1}}
 |D^2u(\xi,\nu)|\leq C.
\end{equation}
The same constant works for every $t\in[0,1]$ and every solution with fixed $\Om$, $\theta$, $\underline u$, and $K$, provided the positive margins in (MRC) are uniform on $\partial\Om$.
\end{theorem}

The proof uses the concavified barrier
\begin{equation*}
 v=\Phi_\mu(u-\underline u)+\tau d-Nd^2,
 \qquad
 \Phi_\mu(s)=\frac{1-e^{-\mu s}}{\mu}.
\end{equation*}
Writing $w=u-\underline u$ and $\cL=a^{ij}D_{ij}+b^kD_k$, the two relevant negative terms in $\cL v$ are
\[
 -\Phi_\mu'(w)(\cJ+\mu Q)-2N a(Dd,Dd).
\]
The two negative quadratic terms control every recession mode except those in Definition \ref{def:ddrs-intro}; (MRC) supplies the missing sign there.  Theorem \ref{thm:MRC-sharp} shows that this condition is necessary and sufficient at the recession level for the fixed smooth zero-order barriers considered in Section~\ref{sec:MRC}.

We next describe the second, independent recession mechanism.  Write $\varphi=u|_{\partial\Om}$.  At $x\in\partial\Om$, decompose
\[
 Du=q+s\nu,\qquad q=D_T\varphi,\qquad s=u_\nu.
\]
The boundary identity is
\begin{equation}\label{eq:tangential-block-intro}
 M(x,s):=D_T^2u=D^2_{\partial\Om}\varphi-s\II_{\partial\Om}.
\end{equation}
For $t\in[0,1]$, put
\begin{equation*}
 C_t=w_t(Du)(I+t^2Du\otimes Du).
\end{equation*}
Its tangential block is $C_{t,T}=w_t(Du)(I_T+t^2q\otimes q)$.
When $M(x,s)\geq0$, define
\begin{equation*}
 L_t(x,s)=\frac\pi2+
 \sum_{\alpha=1}^{n-1}
 \arctan\lambda_\alpha
 \bigl(C_{t,T}^{-1/2}M(x,s)C_{t,T}^{-1/2}\bigr).
\end{equation*}
For $t=1$ we write $\mathcal B(x,s)=L_1(x,s)$.
This is the phase obtained by sending the double-normal Hessian entry to $+\infty$ while keeping the tangential and mixed blocks fixed.

Our second result is an exact identity rather than an asymptotic formula.  Its full derivation is given in Section \ref{sec:normal-gap}.  At a boundary point write
\[
 D^2u=\begin{pmatrix}M&z\\z^T&r_\nu\end{pmatrix},
 \qquad
 C_t=\begin{pmatrix}C_{t,T}&c_t\\c_t^T&c_{t,0}\end{pmatrix},
\]
where $z$ is the mixed tangential-normal block,
$z(\xi)=D^2u(\xi,\nu)$ for $\xi\in T_x\partial\Om$, and
$r_\nu=u_{\nu\nu}$.
The remaining blocks are
\[
 c_t=w_t(Du)t^2sq,
 \qquad c_{t,0}=w_t(Du)(1+t^2s^2).
\]
To shorten the formulas, set
\begin{equation*}
 \begin{gathered}
 D_t=C_{t,T}+\sqrt{-1}M,
 \qquad e_t=c_t+\sqrt{-1}z,
 \qquad S_t=e_t^TD_t^{-1}e_t,\\
 \alpha_t=c_{t,0}-\operatorname{Re}S_t,
 \qquad \beta_t=\operatorname{Im}S_t,
 \qquad \delta_t=L_t-F(A_t[u]).
 \end{gathered}
\end{equation*}
Thus $D_t$ and $e_t$ are the complex tangential block and mixed vector,
respectively, while $S_t$ is their Schur-complement contribution.

The Schur complement gives the following exact identity.

\begin{theorem}\label{thm:exact-gap-intro}
If $D^2u\geq0$, then
\[
 \delta_t=\arctan\frac{\alpha_t}{r_\nu-\beta_t},
 \qquad \alpha_t>0,
 \qquad r_\nu-\beta_t\geq0.
\]
If $|Du|\leq P$ and $|z|\leq Z$, then
\[
 r_\nu=\beta_t+\alpha_t\cot\delta_t,
 \qquad
 1\leq\alpha_t\leq C(P,Z),
 \qquad |\beta_t|\leq C(P,Z).
\]
\end{theorem}

The lower bound $\alpha_t\geq1$ is uniform even when the tangential block is only positive semidefinite.  A consequence is the quantitative double-normal estimate.

\begin{corollary}\label{cor:boundary-hessian-intro}
Under the assumptions of Theorem \ref{thm:mixed-intro}, suppose in addition that
\begin{equation*}
 \mathcal B(x,u_\nu(x))\geq\theta+\sigma
 \qquad\text{on }\partial\Om
\end{equation*}
for some $0<\sigma<\pi/2$.  Then
\begin{equation*}
 \sup_{\partial\Om}|D^2u|
 \leq C_0+C_1\cot\sigma.
\end{equation*}
The same upper bound is uniform along the homotopy if the gap is imposed on a normal-derivative window containing every $u_\nu$.  The constants are independent of both $t$ and $\sigma$.
\end{corollary}

For the original graph equation $t=1$, put
\[
 \delta(x)=\mathcal B(x,u_\nu(x))-\theta
\]
and define the actual boundary gap
\begin{equation}\label{eq:delta-star-intro}
 \delta_*(u)=\min_{x\in\partial\Om}
 \delta(x).
\end{equation}
Every smooth convex solution with finite boundary Hessian has $\delta_*(u)>0$.  The exact boundary identity first controls the double-normal
derivative in terms of the actual phase gap.  A maximum principle for
the mean curvature of a convex constant-phase graph then propagates
this boundary control to the whole graph, yielding the following
global statement.

For nonnegative constants $P,K_\varphi,Z$, we say that a convex solution
$u$ with boundary value $\varphi$ is
\emph{$(P,K_\varphi,Z)$-controlled} if it satisfies
\begin{equation}\label{eq:compact-data-intro}
 \|u\|_{C^1(\overline\Om)}\leq P,
 \qquad
 \|\varphi\|_{C^2(\partial\Om)}\leq K_\varphi,
 \qquad
 \sup_{x\in\partial\Om}
 \sup_{\substack{\xi\in T_x\partial\Om\\|\xi|=1}}
 |D^2u(\xi,\nu)|\leq Z.
\end{equation}

With this terminology, the global estimate takes the following form.

\begin{theorem}\label{thm:compactness-intro}
Every $(P,K_\varphi,Z)$-controlled convex solution of $F(A[u])=\theta$ satisfies
\begin{equation}\label{eq:global-upper-intro}
 \sup_{\overline\Om}|D^2u|+\sup_{\overline\Om}|A[u]|
 \leq C\bigl(1+\cot\delta_*(u)\bigr),
\end{equation}
and
\begin{equation}\label{eq:global-lower-intro}
 \sup_{\overline\Om}|D^2u|
 \geq \cot\delta_*(u)-C.
\end{equation}
Here $C$ depends only on $n,P,K_\varphi,Z$ and the fixed boundary geometry.  Consequently, for any family satisfying \eqref{eq:compact-data-intro} uniformly, 
\begin{equation}\label{eq:compact-equivalence-intro}
 \sup_{\overline\Om}|D^2u_j|\longrightarrow\infty
 \quad\Longleftrightarrow\quad
 \delta_*(u_j)\longrightarrow0.
\end{equation}
The same equivalence holds with $|D^2u_j|$ replaced by the curvature norm $|A[u_j]|$.
\end{theorem}

The gap in \eqref{eq:delta-star-intro} is evaluated at the realized normal derivative.  If a larger prescribed normal window contains slopes that the solution never reaches, the infimum over that window may vanish without curvature blow-up.  Thus \eqref{eq:compact-equivalence-intro} concerns the realized gap rather than an auxiliary window.

Finally, the cotangent rate in Theorems \ref{thm:exact-gap-intro} and \ref{thm:compactness-intro} is optimal.  We have the following smooth sharpness family.

\begin{theorem}\label{thm:radial-intro}
Let $\pi/2\leq\theta<n\pi/2$.
There exist a fixed ball $\Om\subset\R^n$, a constant $\delta_0>0$, and a family of smooth strictly convex solutions
\[
 u_\delta\in C^\infty(\overline\Om),
 \qquad 0<\delta<\delta_0,
\]
of $F(A[u_\delta])=\theta$ such that $\delta_*(u_\delta)=\delta$,
\[
\sup_{0<\delta<\delta_0}
\left(
\|u_\delta\|_{C^1(\overline\Om)}
+\|u_\delta\|_{C^3(\partial\Om)}
+\sup_{x\in\partial\Om}
 \sup_{\substack{\xi\in T_x\partial\Om\\|\xi|=1}}
 |D^2u_\delta(\xi,\nu)|
\right)<\infty,
\]
where the $C^3(\partial\Om)$ norm is the intrinsic norm of the
boundary trace.  Nevertheless,
\[
\sup_{x\in\partial\Om}(u_{\delta})_{\nu\nu}(x)
\asymp\delta^{-1}.
\]
Moreover, the $C^4$ norm of the boundary trace diverges at order $\delta^{-1}$.
\end{theorem}

For $0<\theta<\pi/2$, convexity gives $\mathcal B\geq\pi/2$, so the boundary gap is at least $\pi/2-\theta$.  Hence gap collapse on the convex branch is a genuinely high-phase phenomenon.

The paper is organized as follows.  Section \ref{sec:linearization} records the homotopy, the coefficient bounds, gradient localization, and the failure of joint jet concavity.  Section \ref{sec:MRC} proves the mixed estimate and the recession sharpness of (MRC), and gives stronger but easier-to-check sufficient conditions.  Section \ref{sec:normal-gap} proves the exact Schur-complement identity and the quantitative double-normal estimate.  Section \ref{sec:global} establishes the mean-curvature maximum principle and the global compactness alternative.  Section \ref{sec:radial} contains two sharpness models: a radial family attaining the optimal boundary-gap rate and a rank-loss example showing that a strict lower subsolution does not force strict convexity.

\section{The homotopy and its linearization}\label{sec:linearization}

In this section we record the operator identities used throughout the paper.  They are stated on the closed convex branch whenever possible; approximation gives the same formulas from the strictly convex branch.

The homotopy \eqref{eq:homotopy-intro} is obtained by applying the graph operator to the vertically scaled graph of $tu$ and then dividing its principal curvatures by $t$.  Its first useful property is monotonicity.

\begin{lemma}\label{lem:t-monotone}
If $X\geq0$ and $0\leq s\leq t\leq1$, then
\begin{equation}\label{eq:t-monotone}
 G_t(p,X)\leq G_s(p,X).
\end{equation}
\end{lemma}

\begin{proof}
After an orthogonal change of coordinates, assume $p=|p|e_1$.  Then
\[
 B_t=\diag(w_t^{-1},1,\ldots,1).
\]
Set
\[
 C_{s,t}=\diag(w_s/w_t,1,\ldots,1).
\]
Since $B_t=C_{s,t}B_s$, we have
\begin{equation*}
 A_t(p,X)=\frac{w_s}{w_t}
 C_{s,t}A_s(p,X)C_{s,t}.
\end{equation*}
Both the scalar factor and the symmetric contraction $C_{s,t}$ lie between $0$ and $1$.  The min--max principle therefore gives
$\lambda_j(A_t)\leq\lambda_j(A_s)$ for every $j$.  Since $\arctan$ is increasing, \eqref{eq:t-monotone} follows.
\end{proof}

In particular, a convex strict lower subsolution for $t=1$ remains a strict lower subsolution for the whole homotopy.  The comparison consequence does not require any joint concavity.

\begin{lemma}\label{lem:comparison}
Let $u$ solve \eqref{eq:Pt-intro}.  Suppose that
\begin{equation*}
 G_t(D\underline u,D^2\underline u)\geq\theta+\eta
 \quad\text{in }\Om,
 \qquad \underline u=u\quad\text{on }\partial\Om
\end{equation*}
for some $\eta>0$.  Then $u\geq\underline u$ in $\Om$.
\end{lemma}

\begin{proof}
If $\underline u-u$ had a positive interior maximum at $x_0$, then
\[
 D\underline u(x_0)=Du(x_0),
 \qquad D^2\underline u(x_0)\leq D^2u(x_0).
\]
The gradients agree, so ellipticity in the Hessian variable gives
\[
 \theta+\eta
 \leq G_t(D\underline u,D^2\underline u)
 \leq G_t(Du,D^2u)=\theta,
\]
a contradiction.
\end{proof}

We next record the coefficient bounds on the level set.

Let $(t,p,X)$ be a level jet as in \eqref{eq:level-jet-intro}, and abbreviate
\[
 A=A_t(p,X),\qquad F_A=(I+A^2)^{-1}.
\]
Differentiation in the Hessian variable gives
\begin{equation}\label{eq:a-formula}
 a=G_{t,X}(p,X)=\frac1{w_t}B_tF_AB_t.
\end{equation}
In particular, $a>0$ at every finite jet.

\begin{lemma}\label{lem:coefficient-bounds}
Suppose $p\in K$ and $X\geq0$ with $G_t(p,X)=\theta$.  There are constants $c_0>0$ and $C_0<\infty$, depending only on $n,\theta$ and $K$, such that
\begin{gather}
 c_0\leq\tr a\leq n,\label{eq:trace-a-bounds}\\
 |b|\leq C_0,\label{eq:b-bound}\\
 0\leq a:X
 =\tr\bigl(A(I+A^2)^{-1}\bigr)
 =\sum_{i=1}^n\frac{\kappa_i}{1+\kappa_i^2}
 \leq\frac n2.\label{eq:aX-bound}
\end{gather}
Consequently, in a fixed boundary collar, we have
\begin{equation}\label{eq:J-ell-Q-bounds}
 |\cJ|+|\ell|+Q\leq C_1
\end{equation}
where $C_1$ depends only on $n$, $\theta$, $K$, the
geometry of the fixed boundary collar, and $\|\underline u\|_{C^2}$.
\end{lemma}

\begin{proof}
The upper trace bound follows at once from \eqref{eq:a-formula}, since
$0<B_t\leq I$, $w_t\geq1$, and $0<F_A\leq I$.  For the lower bound, write
\[
 \gamma_i=\arctan\kappa_i\in[0,\pi/2),
 \qquad \sum_i\gamma_i=\theta.
\]
At least one $\gamma_i$ is no larger than $\theta/n$, and therefore
\[
 \tr F_A=\sum_i\cos^2\gamma_i
 \geq\cos^2(\theta/n).
\]
Moreover, $B_t^2\geq w_t^{-2}I$.  Hence
\begin{equation}\label{eq:c0-explicit}
 \tr a
 \geq\frac{\cos^2(\theta/n)}{w_t^3}
 \geq
 \frac{\cos^2(\theta/n)}{(1+\sup_K|p|^2)^{3/2}}
 =:c_0.
\end{equation}

We next estimate the gradient coefficient.  Write
$p=(p_1,\ldots,p_n)$ and, for $r=1,\ldots,n$, put
\[
 E_r=(\partial_{p_r}B_t)B_t^{-1},
 \qquad \omega_r=\partial_{p_r}\log(w_t^{-1}).
\]
On $[0,1]\times K$ these matrices and scalars are uniformly bounded.  A direct differentiation gives
\begin{equation*}
 \partial_{p_r}A
 =\omega_rA+E_rA+AE_r^T.
\end{equation*}
Since $b^r=\partial_{p_r}G_t=F_A:\partial_{p_r}A$, we obtain
\[
 b^r=\omega_r(F_A:A)+F_A:(E_rA)+F_A:(AE_r^T).
\]
The matrix $F_A$ commutes with $A$, and $\|AF_A\|\leq\frac12$,
while $E_r$ and $\omega_r$ are uniformly bounded on $[0,1]\times K$.
Moreover, $0\leq F_A:A\leq n/2$.  Hence every component $b^r$ is
uniformly bounded, and summing over $r$ proves \eqref{eq:b-bound}.  Finally,
$A_t$ is linear in $X$, so Euler's identity gives
\[
 a:X=F_A:A,
\]
which is \eqref{eq:aX-bound}.  The definitions of $\cJ,\ell,Q$, the compactness of $K$, and $0<a\leq I$ now imply \eqref{eq:J-ell-Q-bounds}.
\end{proof}

The coefficient matrix can lose ellipticity in a particular direction only at infinity.

\begin{lemma}\label{lem:finite-coercive}
Fix $R<\infty$ and a boundary collar.  On the compact set of level jets satisfying $|X|\leq R$ there exists $\lambda_R>0$ such that
\begin{equation*}
 a(Dd,Dd)\geq\lambda_R.
\end{equation*}
\end{lemma}

\begin{proof}
The closure of the indicated jet set allows $X\geq0$, but $a$ remains positive definite at every finite $X$ by \eqref{eq:a-formula}.  Since $|Dd|=1$ and all remaining parameters range in compact sets, the conclusion follows by continuity.
\end{proof}

For a family of convex solutions, the compact set $K$ in Definition \ref{def:ddrs-intro} can be determined from a boundary normal window.  Suppose
\begin{equation}\label{eq:normal-window}
 \underline p(x)\leq u_\nu(x)\leq\overline p(x)
 \qquad (x\in\partial\Om).
\end{equation}
Since the tangential gradient equals $D_T\varphi$, set
\begin{equation*}
 K_\partial=
 \left\{D_T\varphi(x)+s\nu(x):
 x\in\partial\Om,
 \ s\in[\underline p(x),\overline p(x)]\right\},
 \qquad
 K=\overline{\operatorname{co}}K_\partial.
\end{equation*}
Here $\operatorname{co}K_\partial$ denotes the convex hull of
$K_\partial$, so $K$ is its closed convex hull.

\begin{lemma}\label{lem:gradient-localization}
If $u\in C^1(\overline\Om)$ and $D^2u\geq0$ in $\Om$, then the boundary window \eqref{eq:normal-window} implies
\begin{equation}\label{eq:gradient-localization}
 Du(\overline\Om)\subset K.
\end{equation}
\end{lemma}

\begin{proof}
Fix an interior point $x$ and a unit vector $e$.  Let $I=(a,b)$ be the
connected component of
\[
 \{s\in\R:x+se\in\Om\}
\]
which contains $0$, and put $y=x+be\in\partial\Om$.  Convexity of $u$
implies that the one-dimensional function
$s\mapsto e\cdot Du(x+se)$ is nondecreasing on $I$.  Therefore
\[
 e\cdot Du(x)\leq e\cdot Du(y)
 \leq\sup_{q\in K_\partial}e\cdot q
 =h_{K_\partial}(e).
\]
Here $h_E(e):=\sup_{q\in E}e\cdot q$ denotes the support function of a set
$E\subset\R^n$. The geometric meaning of this argument is illustrated in
Figure~\ref{fig:gradient-localization}.
This holds for every unit vector $e$.  Since
$h_{K_\partial}=h_{\overline{\operatorname{co}}K_\partial}$, the
support-function characterization of a closed convex set gives
$Du(x)\in K$.  Continuity extends the inclusion to $\overline\Om$, proving
\eqref{eq:gradient-localization}.
\end{proof}

\begin{figure}[htbp]
	\centering
	\includegraphics[width=0.8\textwidth]
	{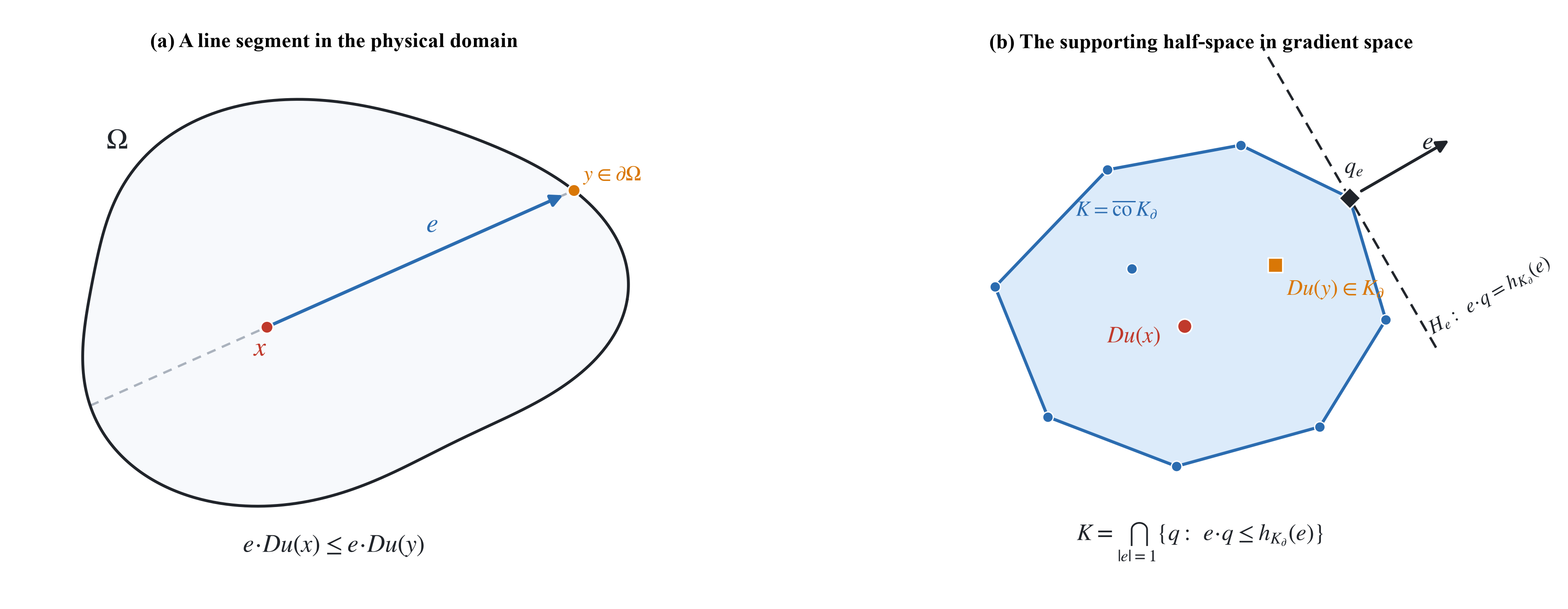}
	\caption{The support-function proof of $Du(x)\in K$.}
	\label{fig:gradient-localization}
\end{figure}

 For example, a strict lower subsolution gives the lower normal bound directly: Lemma \ref{lem:comparison} yields $w=u-\underline u\geq0$, and $w=0$ on the boundary implies the one-sided inequality $w_\nu\geq0$.  An upper bound may come from a fixed supersolution or from any independent first-order estimate.  The mixed recession condition then depends only on this localized set, not on an unspecified global gradient ball.

We finish the section by explaining why a constant-phase subsolution alone is insufficient.  For fixed $p$, the map $X\mapsto G_t(p,X)$ is concave on $X\geq0$.  It is tempting to apply this concavity while both $Du$ and $D^2u$ vary between a solution and a lower subsolution.  The following elementary calculation shows why this is not valid.

\begin{proposition}\label{prop:not-joint-concave}
For $n=2$, $t=1$, and $X=I$, the function
\begin{equation}\label{eq:psi-r}
 \psi(r)=G_1((r,0),I)
 =\arctan(1+r^2)^{-3/2}
  +\arctan(1+r^2)^{-1/2}
\end{equation}
is strictly convex for every $r\geq1$.  In particular, the map $(p,X)\mapsto G_t(p,X)$ is not jointly
concave in general, already for $t=1$.
\end{proposition}

\begin{proof}
For $a>0$ let $y_a(r)=(1+r^2)^{-a}$.  A direct differentiation gives
\begin{equation*}
 \frac{d^2}{dr^2}\arctan y_a
 =\frac{2a(1+r^2)^{-a-2}}{(1+y_a^2)^2}
 \left(A_a+(1+r^2)^{-2a}B_a\right),
\end{equation*}
where
\[
 A_a=(2a+1)r^2-1,
 \qquad B_a=(1-2a)r^2-1.
\]
For $a=1/2$, the bracket equals
$2r^2-1-(1+r^2)^{-1}>0$ when $r\geq1$.  For $a=3/2$, it equals
\[
 4r^2-1-(2r^2+1)(1+r^2)^{-3}>0
 \qquad (r\geq1).
\]
Thus both summands in \eqref{eq:psi-r} have positive second derivative.
\end{proof}

\begin{remark}\label{rem:subsolution-vs-J}
The strict phase inequality
\[
 G_t(D\underline u,D^2\underline u)>G_t(Du,D^2u)
\]
 does not imply $\cJ>0$.  Here is an explicit jet calculation.  For $t=1$, take
 \[
  p=(3,0),\quad X=\diag(1/5,2),
  \qquad q=(1,0),\quad Y=\diag(2,1/10).
 \]
 Direct evaluation gives
 \[
  G_1(q,Y)=\arctan\frac1{\sqrt2}
    +\arctan\frac1{10\sqrt2},
  \qquad
  G_1(p,X)=\arctan\frac1{50\sqrt{10}}
    +\arctan\frac2{\sqrt{10}}.
 \]
 Both sums lie in $(0,\pi/2)$, and their tangents are respectively $\frac{22}{19\sqrt2}$ and $\frac{505}{249\sqrt{10}}$.
 The first is larger, so $G_1(q,Y)>G_1(p,X)$,
 whereas the full linearized separation at $(p,X)$ is
 \[
  G_{1,X}(p,X):(Y-X)+G_{1,p}(p,X)\cdot(q-p)
  =-\frac{14201}{50002\sqrt{10}}<0.
 \]
 Thus even strict phase separation has the wrong supporting sign at this pair of convex jets.  Lemma \ref{lem:comparison} remains valid because the gradients agree at a contact point, but a boundary barrier evaluates the full linearization at points where $Du$ and $D\underline u$ need not agree.  The role of (MRC) is to impose only the part of that full linearized separation which is not supplied by the two available quadratic terms.
\end{remark}

\section{Mixed recession compatibility}\label{sec:MRC}

We now prove Theorem \ref{thm:mixed-intro}.  The proof separates finite or normally elliptic jets from the doubly degenerate recession set; the additional hypothesis is used only in the latter regime.

For later use, we first record the compact form of (MRC).  The following elementary observation also justifies the phrase ``common positive margin'' in \eqref{eq:MRC-seq-intro}.

\begin{lemma}\label{lem:MRC-localization}
Assume (MRC) at $x_0$, and fix a corresponding $\tau>0$.  There exist a neighborhood $U$ of $x_0$ and constants
\begin{equation*}
 R<\infty,\qquad \eta>0,\qquad \zeta>0,
 \qquad \eps>0,
\end{equation*}
such that every level jet with
\begin{equation}\label{eq:MRC-local-region}
 x\in U,\qquad |X|\geq R,\qquad
 a(Dd,Dd)\leq\eta,\qquad Q\leq\zeta
\end{equation}
satisfies
\begin{equation}\label{eq:MRC-local-margin}
 \cJ-\tau\ell\geq4\eps.
\end{equation}
\end{lemma}

\begin{proof}
	Suppose that the conclusion is false.  Choose a decreasing sequence of
	neighborhoods $U_m$ of $x_0$ and positive numbers
	\[
	U_m\downarrow\{x_0\},\qquad
	R_m\uparrow\infty,\qquad
	\eta_m,\zeta_m\downarrow0.
	\]
	Since no uniform positive lower bound is available in any of these
	regions, for every $m$ there is a level jet
	$(x_m,t_m,p_m,X_m)$ such that
	\[
	x_m\in U_m,\qquad |X_m|\geq R_m,\qquad
	a^{(m)}(Dd,Dd)\leq\eta_m,\qquad Q^{(m)}\leq\zeta_m,
	\]
	and
	\[
	\cJ^{(m)}-\tau\ell^{(m)}\leq \frac1m.
	\]
	Consequently,
	\[
	x_m\to x_0,\qquad |X_m|\to\infty,\qquad
	a^{(m)}(Dd,Dd)\to0,\qquad Q^{(m)}\to0,
	\]
	so this is a doubly degenerate recession sequence based at $x_0$.
	
	By Lemma \ref{lem:coefficient-bounds}, the sequence
	$(\cJ^{(m)},-\ell^{(m)})$ is bounded.  After passing to a subsequence,
	\[
	(\cJ^{(m)},-\ell^{(m)})\longrightarrow(j,g)
	\in\cR^*_{x_0}.
	\]
	Taking the limit in the preceding inequality gives
	\[
	j+\tau g
	=\lim_{m\to\infty}
	\bigl(\cJ^{(m)}-\tau\ell^{(m)}\bigr)
	\leq0,
	\]
	contradicting \eqref{eq:MRC-intro}.  Hence there is a positive uniform
	margin $c_*$ on some region of the form
	\eqref{eq:MRC-local-region}.  Taking
	$\eps=c_*/4$ yields \eqref{eq:MRC-local-margin}.
\end{proof}

We next construct the concavified local barrier.

Let $u$ be as in Theorem \ref{thm:mixed-intro} and put
\begin{equation*}
 w=u-\underline u\geq0.
\end{equation*}
Since $w=0$ on the boundary and $|Dw|$ is uniformly bounded, integration
along the inward normal segment from the nearest boundary point gives
\begin{equation}\label{eq:w-less-d}
 0\leq w\leq C d
\end{equation}
in a fixed collar.  For $\mu\geq0$, define
\begin{equation*}
 \Phi_\mu(s)=
 \begin{cases}
  (1-e^{-\mu s})/\mu,&\mu>0,\\
  s,&\mu=0.
 \end{cases}
\end{equation*}
Thus
\begin{equation*}
 \Phi_\mu'(s)=e^{-\mu s}>0,
 \qquad
 \Phi_\mu''(s)=-\mu e^{-\mu s}\leq0.
\end{equation*}

Along a solution, let
\begin{equation*}
 \cL=a^{ij}\partial_{ij}+b^k\partial_k
\end{equation*}
be the linearized operator.  At the solution jet,
\begin{equation*}
 \cL w=-\cJ,
 \qquad a(Dw,Dw)=Q.
\end{equation*}

\begin{lemma}\label{lem:strict-local-barrier}
Fix $x_0\in\partial\Om$ and suppose (MRC) holds there.  There are positive constants $\mu,N,\delta,c$ and a neighborhood $U$ of $x_0$ such that
\begin{equation}\label{eq:v-barrier}
 v=\Phi_\mu(w)+\tau d-Nd^2
\end{equation}
satisfies
\begin{equation}\label{eq:Lv-negative}
 \cL v\leq-c
 \quad\text{in }U\cap\{0<d<\delta\},
\end{equation}
and
\begin{equation*}
 v=0\quad\text{on }U\cap\partial\Om,
 \qquad
 v\geq0\quad\text{when }d=\delta,
\end{equation*}
The same constants may be used for any family with fixed data bounds and fixed constants in Lemma \ref{lem:MRC-localization}.
\end{lemma}

\begin{proof}
	Let $\tau$ be supplied by (MRC), and take
	$U,R,\eta,\zeta,\eps$ from Lemma
	\ref{lem:MRC-localization}.  Put
	\[
	w=u-\underline u,\qquad E=e^{-\mu w}.
	\]
	Since $\cL w=-\cJ$, $a(Dw,Dw)=Q$, and
	$\Phi_\mu''=-\mu\Phi_\mu'$, differentiating
	\eqref{eq:v-barrier} gives
	\begin{align}
		\cL v
		&=-E(\cJ+\mu Q)+\tau\ell
		-2Nd\,\ell-2Na(Dd,Dd)\notag\\
		&=-E\bigl(\cJ-\tau\ell+\mu Q\bigr)
		+\tau(1-E)\ell
		-2Nd\,\ell-2Na(Dd,Dd).
		\label{eq:Lv-exact}
	\end{align}
	
	We first choose $\mu$.  By Lemma
	\ref{lem:coefficient-bounds}, there is $C_2>0$ such that $\cJ-\tau\ell\geq-C_2$
	on all relevant level jets.  Since $\zeta>0$, we may fix $\mu$ so
	large that
	\begin{equation}\label{eq:mu-choice}
		-C_2+\mu\zeta\geq4\eps.
	\end{equation}
	We divide the level jets in the local collar into three regions.
	
	First, suppose that
	\[
	|X|\geq R,\qquad
	a(Dd,Dd)\leq\eta,\qquad Q\leq\zeta.
	\]
	Then Lemma \ref{lem:MRC-localization} gives
	\[
	\cJ-\tau\ell+\mu Q
	\geq\cJ-\tau\ell\geq4\eps.
	\]
	Second, if $Q\geq\zeta$, then \eqref{eq:mu-choice} gives
	\[
	\cJ-\tau\ell+\mu Q
	\geq-C_2+\mu\zeta\geq4\eps.
	\]
	Thus the first term on the second line of
	\eqref{eq:Lv-exact} is strictly negative in both regions.
	
	It remains to consider the region in which $Q<\zeta$ and either
	$|X|<R$ or $a(Dd,Dd)>\eta$.  In the first case, Lemma
	\ref{lem:finite-coercive} gives
	$a(Dd,Dd)\geq\lambda_R$; in the second case it is already larger
	than $\eta$.  Hence throughout this region,
	\begin{equation}\label{eq:lambda-star}
		a(Dd,Dd)\geq
		\lambda_*:=\min\{\lambda_R,\eta\}>0.
	\end{equation}
	Since $\ell$ is uniformly bounded, shrink the collar, independently
	of $N$, so that $2d|\ell|\leq\lambda_*$.
	The two terms involving $N$ then satisfy
	\[
	-2Nd\,\ell-2Na(Dd,Dd)
	\leq2Nd|\ell|-2N\lambda_*
	\leq-N\lambda_*.
	\]
	Once $\mu$ is fixed, all the remaining terms in
	\eqref{eq:Lv-exact} are uniformly bounded.  We may therefore choose
	$N$ sufficiently large so that
	\[
	\cL v\leq-1
	\]
	throughout the region \eqref{eq:lambda-star}.
	
	We now return to the first two regions.  By
	\eqref{eq:w-less-d}, after shrinking the collar to
	$\{0<d<\delta\}$, we have
	\[
	E\geq\frac12,
	\qquad
	0\leq1-E\leq\mu w\leq C\mu d.
	\]
	Since the principal bracket is at least $4\eps$,
	\[
	-E\bigl(\cJ-\tau\ell+\mu Q\bigr)\leq-2\eps.
	\]
	After $N$ has been fixed, decrease $\delta$ further so that
	\[
	\left|\tau(1-E)\ell-2Nd\,\ell\right|\leq\eps.
	\]
	The remaining term $-2Na(Dd,Dd)$ is nonpositive, and hence
	\[
	\cL v\leq-\eps
	\]
	in the first two regions.  Taking $c=\min\{1,\eps\}$ proves
	\eqref{eq:Lv-negative} throughout the local collar.
	
	Finally, require $\delta\leq\tau/N$.  On the true boundary,
	$w=d=0$, so $v=0$.  On the artificial boundary $d=\delta$,
	\[
	v=\Phi_\mu(w)+\tau\delta-N\delta^2
	\geq\delta(\tau-N\delta)\geq0.
	\]
	This completes the proof.
\end{proof}

We complete the proof of Theorem \ref{thm:mixed-intro}.  The use of a Euclidean Killing field is convenient because it avoids third derivatives of the unknown solution.

\begin{proof}[Proof of Theorem \ref{thm:mixed-intro}]
Fix $x_0\in\partial\Om$, translate it to the origin, and choose coordinates so that $e_n=\nu(x_0)$.  Locally write
\[
 \partial\Om=\{x_n=\rho(x')\},
 \qquad \rho(0)=|D\rho(0)|=0.
\]
For $1\leq\alpha<n$, define the affine vector field
\begin{equation}\label{eq:Killing-field}
 T_\alpha
 =\partial_\alpha+
 \sum_{\beta<n}\rho_{\alpha\beta}(0)
 \bigl(x_\beta\partial_n-x_n\partial_\beta\bigr).
\end{equation}
It is the sum of a translation and infinitesimal Euclidean rotations.  Since $G_t$ is invariant under rigid motions of the base variables, differentiating this invariance gives
\begin{equation*}
 \cL(T_\alpha u)=0.
\end{equation*}
Consequently,
\begin{equation}\label{eq:LTw-bound}
 \bigl|\cL\bigl(T_\alpha(u-\underline u)\bigr)\bigr|
 =|\cL(T_\alpha\underline u)|\leq C.
\end{equation}
Indeed, $T_\alpha$ has affine coefficients, $\underline u$ is fixed in $C^3$, and \eqref{eq:trace-a-bounds}--\eqref{eq:b-bound} control the contraction.

The vector field \eqref{eq:Killing-field} is tangent to the quadratic osculating graph of $\partial\Om$ at the origin.  Since $u-\underline u=0$ on the true boundary and both gradients are bounded,
\begin{equation}\label{eq:Tw-boundary}
 \bigl|T_\alpha(u-\underline u)\bigr|
 \leq C|x|^2
 \qquad\text{on }\partial\Om
\end{equation}
after reducing the coordinate neighborhood.

Let $v$ be the barrier from Lemma \ref{lem:strict-local-barrier}.  In a small local collar set, consider
\begin{equation*}
 W_\pm
 =A v+B|x|^2
 \pm T_\alpha(u-\underline u).
\end{equation*}
The linearization of $|x|^2$ is uniformly bounded.  First choose $B$ so that \eqref{eq:Tw-boundary}, the artificial face $d=\delta$, and the lateral boundary of the coordinate neighborhood are controlled.  Then choose $A\gg B+1$.  Equations \eqref{eq:Lv-negative} and \eqref{eq:LTw-bound} give
\begin{equation*}
 \cL W_\pm\leq0,
 \qquad
 W_\pm\geq0\quad\text{on the boundary of the local collar}.
\end{equation*}
The minimum principle yields $W_\pm\geq0$ inside.  Both functions vanish at $x_0$.  Taking their inward normal derivatives gives
\begin{equation*}
 \left|\partial_\nu T_\alpha(u-\underline u)(x_0)\right|
 \leq A\,\partial_\nu v(x_0)
 =A\bigl((u-\underline u)_\nu(x_0)+\tau\bigr)\leq C.
\end{equation*}
At $x_0$, differentiating the coefficients of $T_\alpha$ shows
\[
 \partial_\nu T_\alpha(u-\underline u)
 =D^2(u-\underline u)(e_\alpha,\nu)+O(|D(u-\underline u)|).
\]
All lower-order terms are bounded.  Hence $|D^2u(e_\alpha,\nu)(x_0)|\leq C$.  The strict MRC margin is stable in a neighborhood of each boundary point; compactness of $\partial\Om$ provides a finite cover and a uniform constant.
\end{proof}

We next ask how much of (MRC) is forced by the barrier argument.

 Consider a fixed smooth local barrier near $x_0$ of the form
 \begin{equation}\label{eq:general-core-barrier}
 	\mathcal V(w,x)=\Psi\bigl(w,d(x)\bigr)+\chi(x),
 \end{equation}
 where $w=u-\underline u$, $\Psi$ is smooth near $(0,0)$, and
 \begin{equation*}
 	\Psi(0,0)=0,\qquad
 	c_w:=\Psi_w(0,0)>0,\qquad
 	c_d:=\Psi_d(0,0)\geq0.
 \end{equation*}
 We assume that $\chi$ is smooth in a full neighborhood of $x_0$ and
 has an ambient local minimum there.  After subtracting a constant,
 we may arrange that
 \begin{equation*}
 	\chi(x_0)=0,\qquad D\chi(x_0)=0,\qquad D^2\chi(x_0)\geq0.
 \end{equation*}
 This class contains the standard tangential correction
 $\chi(x)=B|x-x_0|^2$, with $B\geq0$.
 
 At a solution jet, the chain rule and the identities
 \[
 \cL w=-\cJ,\qquad
 \cL d=\ell,\qquad
 Dw=p-D\underline u
 \]
 give
 \begin{align*}
 	\cL\mathcal V(w,x)
 	={}&-\Psi_w(w,d)\cJ+\Psi_d(w,d)\ell
 	+\Psi_{ww}(w,d)Q\\
 	&+2\Psi_{wd}(w,d)a(p-D\underline u,Dd)
 	+\Psi_{dd}(w,d)a(Dd,Dd)\\
 	&+a:D^2\chi+b\cdot D\chi.
 \end{align*}
 To study this expression independently of a particular solution, let
 $s$ represent a possible value of $w(x)$.  For a level jet
 $(t,p,X)$ at $x$, define its algebraic action on the barrier by
 \begin{align}
 	\cL_{(t,p,X)}\mathcal V[s]
 	:={}&-\Psi_w(s,d)\cJ+\Psi_d(s,d)\ell
 	+\Psi_{ww}(s,d)Q\notag\\
 	&+2\Psi_{wd}(s,d)a(p-D\underline u,Dd)
 	+\Psi_{dd}(s,d)a(Dd,Dd)\notag\\
 	&+a:D^2\chi+b\cdot D\chi.
 	\label{eq:algebraic-barrier-linearization}
 \end{align}
 Here $d=d(x)$, and every derivative of $\Psi$ on the right hand side
 is evaluated at $(s,d(x))$.
 
 Fix $C_{\rm col}>0$ such that
 \begin{equation*}
 	0\leq w(x)\leq C_{\rm col}d(x)
 \end{equation*}
 in the boundary collar for the family under consideration.  We say
 that $\mathcal V$ has \emph{uniform recession negativity} at $x_0$ if
 there exist a neighborhood $U$ of $x_0$ and constants
 \[
 R<\infty,\qquad \eta,\zeta,c>0,
 \]
 such that
 \begin{equation*}
 	\cL_{(t,p,X)}\mathcal V[s]\leq-c
 \end{equation*}
 whenever $0\leq s\leq C_{\rm col}d(x)$ and
 \begin{equation*}
 		x\in U,\qquad |X|\geq R,\qquad
 		a(Dd,Dd)\leq\eta,\qquad Q\leq\zeta.
 \end{equation*}
 Thus the barrier remains uniformly strict on all sufficiently large
 level jets for which both available quadratic coercive terms are
 small.

The preceding barrier class characterizes (MRC) at the recession level.

\begin{theorem}\label{thm:MRC-sharp}
There exists a core barrier of the form \eqref{eq:general-core-barrier} with uniform recession negativity at $x_0$ if and only if (MRC) holds at $x_0$.  Any such barrier forces
\begin{equation}\label{eq:necessary-slope}
 \inf_{(j,g)\in\cR^*_{x_0}}
 \left(j+\frac{c_d}{c_w}g\right)>0
\end{equation}
when $c_d>0$; if $c_d=0$, the same conclusion holds with a sufficiently small positive slope.  Conversely, the barrier \eqref{eq:v-barrier} realizes any slope supplied by (MRC).
\end{theorem}

\begin{proof}
 At the solution jet, the chain rule gives
\begin{align}
 \cL\Psi(w,d)
 ={}&-\Psi_w\cJ+\Psi_d\ell
 +\Psi_{ww}a(Dw,Dw)\notag\\
 &+2\Psi_{wd}a(Dw,Dd)
 +\Psi_{dd}a(Dd,Dd).
 \label{eq:L-general-Psi}
\end{align}
 Along a doubly degenerate recession sequence, choose any
 $s_m\in[0,C_{\rm col}d(x_m)]$.  Then $s_m\to0$, and the quadratic terms in \eqref{eq:algebraic-barrier-linearization} satisfy
\[
 a(p-D\underline u,p-D\underline u)=Q\to0,
 \qquad a(Dd,Dd)\to0.
\]
Since $a\geq0$, Cauchy--Schwarz gives
\begin{equation*}
 |a(p-D\underline u,Dd)|^2\leq Q\,a(Dd,Dd)\to0.
\end{equation*}
 Since $x_0$ is an ambient local minimum of $\chi$, we have $D\chi(x_0)=0$ and $D^2\chi(x_0)\geq0$.  The coefficient bounds, together with $a(Dd,Dd)\to0$, show that the normal and mixed contractions vanish; the surviving tangential contraction is nonnegative.  Thus $\chi$ cannot improve a required negative upper bound.  Passing to a recession limit pair $(j,g)$ in \eqref{eq:L-general-Psi}, or equivalently in \eqref{eq:algebraic-barrier-linearization}, gives
\begin{equation*}
 \liminf_{m\to\infty}
  \cL_{(t_m,p_m,X_m)}\mathcal V[s_m]
 \geq-c_wj-c_dg.
\end{equation*}
On the other hand, uniform recession negativity gives
\[
 -c\geq\limsup_{m\to\infty}
 \cL_{(t_m,p_m,X_m)}\mathcal V[s_m].
\]
Combining the last two inequalities yields
\[
 c_wj+c_dg\geq c,
 \qquad
 j+\frac{c_d}{c_w}g\geq\frac c{c_w}>0,
\]
uniformly over all recession limit pairs.  This is
\eqref{eq:necessary-slope}.  If $c_d=0$, the same argument gives a uniform
positive lower bound for $j$; boundedness of $g$ then permits a sufficiently
small positive slope.  This proves necessity.  For sufficiency, take
$\chi\equiv0$ and use the barrier from Lemma \ref{lem:strict-local-barrier}.
\end{proof}

\begin{remark}\label{rem:not-PDE-necessity}
Theorem \ref{thm:MRC-sharp} does not assert that failure of (MRC) produces an actual sequence of solutions with unbounded mixed derivatives.  The bad level jets may not be dynamically reachable, and a higher-order or equation-specific construction may lie outside the fixed zero-order class \eqref{eq:general-core-barrier}.  The theorem concerns only the standard maximum-principle barrier argument.
\end{remark}

The abstract recession condition has several more concrete sufficient criteria.  They are useful when verifying a theorem in examples.

\begin{proposition}\label{prop:MRC-criteria}
Each of the following conditions implies (MRC) at $x_0$.
\begin{enumerate}[label=\textup{(\roman*)}]
\item There is $\eps_0>0$ such that $\cJ\geq\eps_0$ on all level jets near $x_0$.
\item There is $\eps_0>0$ such that $j\geq\eps_0$ for all $(j,g)\in\cR^*_{x_0}$.
\item There is $\gamma_0>0$ such that $g\geq\gamma_0$ for all $(j,g)\in\cR^*_{x_0}$.
\end{enumerate}
Condition \textup{(i)} is the full jet-subsolution condition; \textup{(ii)} is its asymptotic version; and \textup{(iii)} is the pure drift-curvature condition.
\end{proposition}

\begin{proof}
The quantity $g=-\ell$ is uniformly bounded.  Thus (i) or (ii) implies (MRC) after choosing $\tau>0$ sufficiently small.  The quantity $j=\cJ$ has a uniform lower bound, so (iii) implies (MRC) after choosing $\tau$ sufficiently large.
\end{proof}

The first criterion follows from a familiar phase separation if the drift loss is quantitatively dominated.  This is stronger than (MRC), because it is imposed on every level jet rather than only on doubly degenerate recession jets, but it is often easier to verify directly.

\begin{corollary}\label{cor:checkable-JS}
Let
\[
 C_b=\sup_{[0,1]\times K\times\{G_t=\theta\}}|G_{t,p}|,
 \qquad
 D_0=\sup_{x\in\overline\Om,\,p\in K}|D\underline u(x)-p|.
\]
Suppose
\begin{equation*}
 \inf_{x,t,p\in K}G_t(p,D^2\underline u(x))
 \geq\theta+\eta_0,
 \qquad \eta_0>C_bD_0.
\end{equation*}
Then $\cJ\geq\eta_0-C_bD_0>0$ on every level jet, and hence (MRC) holds.
\end{corollary}

\begin{proof}
For fixed $p$, concavity in $X$ gives
\[
 G_{t,X}(p,X):(D^2\underline u-X)
 \geq G_t(p,D^2\underline u)-G_t(p,X)
 \geq\eta_0.
\]
The drift term satisfies
\[
 b\cdot(D\underline u-p)
 \geq-|b|\,|D\underline u-p|
 \geq-C_bD_0.
\]
Adding this to the Hessian estimate gives
$\cJ\geq\eta_0-C_bD_0>0$, as claimed.
\end{proof}

There is also a simple purely geometric condition which avoids a lower subsolution in the mixed estimate.

\begin{corollary}\label{cor:Serrin-coarse}
Assume $|Du|\leq P$, and let $c_0$ be the lower trace bound in
\eqref{eq:c0-explicit}.  Let $C_b$ be the supremum of $|b|$ over the
corresponding level jets.  If
\begin{equation}\label{eq:Serrin-coarse}
 \II_{\partial\Om}\geq\kappa_0 I,
 \qquad \kappa_0>\frac{2C_b}{c_0},
\end{equation}
then the mixed derivative estimate \eqref{eq:mixed-est-intro} follows from a pure distance barrier and a fixed extension of the boundary data.
\end{corollary}

\begin{proof}
 Choose first
 \[
  \frac{2C_b}{c_0}<\kappa_1<\kappa_0.
 \]
 By continuity of the second fundamental forms of the parallel
 hypersurfaces, the collar can then be shrunk so that every level set of $d$
 has second fundamental form at least $\kappa_1I$.  At every point of this
 collar,
 \[
  \ell\leq-\kappa_1\tr a_T+C_b.
 \]
 If $\tr a_T\geq c_0/2$, then
 \[
  \ell\leq-\frac{\kappa_1c_0}{2}+C_b<0.
 \]
 If $\tr a_T<c_0/2$, then, because
 $\tr a=\tr a_T+a(Dd,Dd)$ in the tangential-normal splitting,
 \eqref{eq:trace-a-bounds} gives
 \[
  a(Dd,Dd)\geq\frac{c_0}{2},
 \]
 and the term $-2Na(Dd,Dd)$ is coercive.  For $v=\tau d-Nd^2$, we have
 \[
 \cL v=(\tau-2Nd)\ell-2Na(Dd,Dd).
 \]
 Choose $N$ large enough to dominate the bounded term $\tau\ell$, and
 then shrink the collar so that the term $2Nd\ell$ is negligible.  The
 two cases above then give $\cL v\leq-c<0$.  Since $v=0$ on
 $\partial\Om$ and $v\geq0$ on the artificial boundary, the
 Killing-field comparison applied with a fixed smooth extension of
 $\varphi$ yields the mixed derivative estimate.
\end{proof}

\begin{remark}\label{rem:MRC-weakness}
The coarse criterion \eqref{eq:Serrin-coarse} is not expected to be sharp because it estimates the drift by its absolute supremum.  Definition \ref{def:MRC-intro} retains the actual coupling between the limiting coefficient matrix, the boundary second fundamental form, the drift, and the subsolution separation.  It may hold even when neither the asymptotic jet-subsolution condition nor the pure geometric condition holds by itself.
\end{remark}

\section{The exact boundary-gap identity}\label{sec:normal-gap}

We turn to the double-normal direction.  The phase difference admits an exact formula, which gives the optimal dependence on the gap without a contradiction sequence.

We first derive the boundary identity by a complex Schur complement.

Fix a boundary point and use the Euclidean orthogonal splitting
\[
 T_x\partial\Om\oplus\R\nu.
\]
Write the Euclidean Hessian in tangential-normal blocks as
\begin{equation*}
 D^2u=
 \begin{pmatrix}
  M&z\\ z^T&r_\nu
 \end{pmatrix},
 \qquad r_\nu=u_{\nu\nu}.
\end{equation*}
Here $M=D_T^2u$, $z=D^2u(\,\cdot\,,\nu)|_{T\partial\Om}$ is the mixed
block, and $r_\nu$ is the double-normal entry.  Let $Du=q+s\nu$, where
$q=D_T\varphi$ and $s=u_\nu$.  Using the shorthand $w_t=w_t(Du)$, we have
\[
 A_t[u]=C_t^{-1/2}D^2u\,C_t^{-1/2},
 \qquad C_t=w_t(I+t^2Du\otimes Du).
\]
The corresponding block decomposition is
\begin{equation*}
 C_t=
 \begin{pmatrix}
  C_{t,T}&c_t\\c_t^T&c_{t,0}
 \end{pmatrix},
\end{equation*}
where
\begin{equation}\label{eq:C-block-explicit}
 C_{t,T}=w_t(I_T+t^2q\otimes q),
 \qquad c_t=w_tt^2sq,
 \qquad c_{t,0}=w_t(1+t^2s^2).
\end{equation}
Recall that
\begin{equation*}
 D_t=C_{t,T}+\sqrt{-1} M,
 \qquad e_t=c_t+\sqrt{-1} z,
 \qquad S_t=e_t^TD_t^{-1}e_t.
\end{equation*}
The transpose in the last expression is the complex bilinear transpose, as required by the determinant formula; it is not a Hermitian product.

Let
\begin{equation*}
 \Phi_t(M)
 =\sum_{\alpha=1}^{n-1}
 \arctan\lambda_\alpha(C_{t,T}^{-1/2}MC_{t,T}^{-1/2}),
 \qquad L_t=\frac\pi2+\Phi_t(M).
\end{equation*}

\begin{lemma}\label{lem:exact-Schur}
If $D^2u\geq0$, then
\begin{equation}\label{eq:exact-Schur}
 L_t-F(A_t[u])
 =\arctan\frac{c_{t,0}-\operatorname{Re}S_t}
 {r_\nu-\operatorname{Im}S_t}.
\end{equation}
The quotient is interpreted as $+\infty$ if the denominator vanishes.  Moreover,
\begin{equation}\label{eq:s-quadrant}
 c_{t,0}-\operatorname{Re}S_t>0,
 \qquad r_\nu-\operatorname{Im}S_t\geq0,
\end{equation}
and the second inequality is strict if $D^2u>0$.
\end{lemma}

\begin{proof}
Since
\begin{equation}\label{eq:I-iA-factor}
 I+\sqrt{-1} A_t
 =C_t^{-1/2}(C_t+\sqrt{-1}D^2u)C_t^{-1/2},
\end{equation}
the two determinants have the same phase.  The block determinant formula gives
\begin{equation}\label{eq:block-det}
 \det(C_t+\sqrt{-1}D^2u)
 =\det(C_{t,T}+\sqrt{-1}M)\,\mathfrak s_t,
\end{equation}
where
\begin{equation*}
 \begin{aligned}
 \mathfrak s_t
 &=c_{t,0}+\sqrt{-1}r_\nu
 -(c_t+\sqrt{-1}z)^T(C_{t,T}+\sqrt{-1}M)^{-1}
 (c_t+\sqrt{-1}z)\\
 &=(c_{t,0}-\operatorname{Re}S_t)
   +\sqrt{-1}(r_\nu-\operatorname{Im}S_t).
 \end{aligned}
\end{equation*}

We first locate $\mathfrak s_t$.  The lower-right entry in the inverse block formula is $\mathfrak s_t^{-1}$.  Let
\[
 y=(C_t+\sqrt{-1}D^2u)^{-1}e_n.
\]
As $e_n$ is real,
\begin{equation*}
 \mathfrak s_t^{-1}=e_n^*y=((C_t+\sqrt{-1}D^2u)y)^*y
 =y^*C_ty-\sqrt{-1} y^*D^2u\,y.
\end{equation*}
Thus $\operatorname{Re}\mathfrak s_t^{-1}>0$ and
$\operatorname{Im}\mathfrak s_t^{-1}\leq0$.  Therefore $\mathfrak s_t$ lies
in the closed first quadrant, with strictly positive real part.  If
$D^2u>0$, its imaginary part is also positive.  This proves
\eqref{eq:s-quadrant}.

 For a positive semidefinite matrix $N$, write
 $\operatorname{Arg}_0\det(I+\sqrt{-1} N)$ for the continuous lift of the determinant phase along $\eps N$, $0\leq\eps\leq1$, starting from zero.  Then
 \[
  \operatorname{Arg}_0\det(I+\sqrt{-1} N)
  =\sum_j\arctan\lambda_j(N)
 \]
 on this lifted branch.  Applying the same lift to \eqref{eq:I-iA-factor} and \eqref{eq:block-det} gives
\begin{equation*}
 F(A_t[u])=\Phi_t(M)+\arg\mathfrak s_t.
\end{equation*}
 Indeed, after replacing $D^2u$ by $\eps D^2u$, all determinant phases in the block identity start from zero at $\eps=0$; hence their lifted phases add exactly, with no $2\pi$ ambiguity.  Since $\mathfrak s_t$ is in the first quadrant,
\[
 L_t-F(A_t[u])
 =\frac\pi2-\arg\mathfrak s_t
 =\arctan\frac{\operatorname{Re}\mathfrak s_t}
 {\operatorname{Im}\mathfrak s_t},
\]
which is \eqref{eq:exact-Schur}.
\end{proof}

Retain the notation $\alpha_t$ and $\beta_t$ introduced in the statement of
Theorem \ref{thm:exact-gap-intro}.
The positivity in Lemma \ref{lem:exact-Schur} has a quantitative form.

\begin{lemma}\label{lem:alpha-beta}
Assume $D^2u\geq0$, $|Du|\leq P$, and $|z|\leq Z$.  Let
\begin{equation*}
 w_*=(1+P^2)^{1/2},
 \qquad
 \Xi(P,Z)=\frac{w_*P^4}{4}+Z^2.
\end{equation*}
Then
\begin{equation}\label{eq:alpha-beta-bounds}
 1\leq\alpha_t
 \leq w_*(1+P^2)+\Xi(P,Z),
 \qquad
 |\beta_t|\leq\Xi(P,Z).
\end{equation}
\end{lemma}

\begin{proof}
Put
\[
 \widehat M=C_{t,T}^{-1/2}MC_{t,T}^{-1/2},
 \qquad \widehat c=C_{t,T}^{-1/2}c_t,
 \qquad \widehat z=C_{t,T}^{-1/2}z.
\]
Here $\widehat M$ is a real symmetric nonnegative $(n-1)\times(n-1)$
matrix and $\widehat c,\widehat z\in\R^{n-1}$.  Since
$D_t=C_{t,T}^{1/2}(I+\sqrt{-1}\widehat M)C_{t,T}^{1/2}$ and
$e_t=C_{t,T}^{1/2}(\widehat c+\sqrt{-1}\widehat z)$, the identity
$S_t=e_t^TD_t^{-1}e_t$ becomes
\begin{equation*}
 S_t=(\widehat c+\sqrt{-1}\widehat z)^T
 (I+\sqrt{-1}\widehat M)^{-1}
 (\widehat c+\sqrt{-1}\widehat z).
\end{equation*}
Diagonalize $\widehat M$.  If $k_a\geq0$ are its eigenvalues, then for each component
\begin{equation*}
 \operatorname{Re}
 \frac{(\widehat c_a+\sqrt{-1}\widehat z_a)^2}
 {1+\sqrt{-1} k_a}
 =\widehat c_a^2-
 \frac{(\widehat z_a-k_a\widehat c_a)^2}{1+k_a^2}
 \leq\widehat c_a^2.
\end{equation*}
Therefore
\begin{equation*}
 \operatorname{Re}S_t\leq c_t^TC_{t,T}^{-1}c_t.
\end{equation*}
Using \eqref{eq:C-block-explicit}, a direct computation yields
\begin{align*}
 \alpha_t
 &\geq c_{t,0}-c_t^TC_{t,T}^{-1}c_t\\
 &=w_t\left(1+
  \frac{t^2s^2}{1+t^2|q|^2}\right)
 =\frac{w_t^3}{1+t^2|q|^2}
 \geq1.
\end{align*}

 The complex matrix $D_t=C_{t,T}+\sqrt{-1}M$ has real part bounded below by $w_tI$.  More precisely, for every complex vector $x$,
 \[
  w_t|x|^2\leq x^*C_{t,T}x
  =\operatorname{Re}\bigl(x^*D_tx\bigr)
  \leq |D_tx|\,|x|.
 \]
 Hence
\begin{equation*}
 \|D_t^{-1}\|\leq w_t^{-1}.
\end{equation*}
It follows that
\begin{align}
 |S_t|
 &\leq\frac{|c_t+\sqrt{-1} z|^2}{w_t}
 =w_tt^4s^2|q|^2+\frac{|z|^2}{w_t}\notag\\
 &\leq\frac{w_*P^4}{4}+Z^2=\Xi(P,Z).
 \label{eq:S-upper}
\end{align}
Finally, $c_{t,0}\leq w_*(1+P^2)$.  Combining this with \eqref{eq:S-upper} proves \eqref{eq:alpha-beta-bounds}.
\end{proof}

\begin{proof}[Proof of Theorem \ref{thm:exact-gap-intro}]
Set $\delta_t=L_t-F(A_t[u])$.  Lemma \ref{lem:exact-Schur} gives
\[
 \delta_t=\arctan\frac{\alpha_t}{r_\nu-\beta_t},
 \qquad
 \alpha_t>0,
 \qquad r_\nu-\beta_t\geq0.
\]
Thus $0<\delta_t\leq\pi/2$.  If $\delta_t<\pi/2$, taking the tangent and
solving for the double-normal entry gives the exact pointwise relation
\begin{equation}\label{eq:R-alpha-beta}
 r_\nu=\beta_t+\alpha_t\cot\delta_t.
\end{equation}
If $\delta_t=\pi/2$, the denominator in the quotient vanishes, so
$r_\nu=\beta_t$ and the same formula remains valid because
$\cot(\pi/2)=0$.  Finally, Lemma \ref{lem:alpha-beta} yields
\[
 1\leq\alpha_t\leq C(P,Z),
 \qquad |\beta_t|\leq C(P,Z),
\]
which proves every assertion of the theorem.
\end{proof}

The limiting phase for $t\leq1$ is no smaller than the original one.

\begin{lemma}\label{lem:L-monotone}
If $M\geq0$, then
\begin{equation*}
 L_t(x,s)\geq L_1(x,s)=\mathcal B(x,s)
 \qquad (0\leq t\leq1).
\end{equation*}
\end{lemma}

\begin{proof}
From \eqref{eq:C-block-explicit},
\[
 C_{t,T}=w_t(I_T+t^2q\otimes q)
\]
is nondecreasing in the Loewner order as a function of $t$.  Therefore
$C_{t,T}\leq C_{1,T}$.  The generalized min--max principle for the pair
$(M,C_{t,T})$ gives
\[
 \lambda_\alpha(C_{t,T}^{-1/2}MC_{t,T}^{-1/2})
 \geq
 \lambda_\alpha(C_{1,T}^{-1/2}MC_{1,T}^{-1/2}).
\]
Summing the increasing function $\arctan$ proves the claim.
\end{proof}

\begin{proposition}\label{prop:double-normal}
Suppose
\begin{equation*}
 D^2u\geq0,
 \qquad F(A_t[u])=\theta,
 \qquad \mathcal B(x,s)\geq\theta+\sigma,
\end{equation*}
where $s=u_\nu(x)$ and $0<\sigma<\pi/2$.  If $|Du|\leq P$ and $|z|\leq Z$, then
\begin{equation*}
 r_\nu\leq\Xi(P,Z)+
 \bigl[w_*(1+P^2)+\Xi(P,Z)\bigr]\cot\sigma.
\end{equation*}
In particular, $r_\nu\leq C(P,Z)(1+\sigma^{-1})$.
\end{proposition}

\begin{proof}
By Lemma \ref{lem:L-monotone},
\[
 \delta_t=L_t-\theta\geq\mathcal B-\theta\geq\sigma.
\]
 Apply \eqref{eq:R-alpha-beta}, the upper bounds in Lemma
 \ref{lem:alpha-beta}, and the monotonicity of $\cot$ on $(0,\pi/2)$.  At
 the endpoint $\delta_t=\pi/2$, formula \eqref{eq:R-alpha-beta} gives
 $r_\nu=\beta_t$ directly.
\end{proof}

At $t=1$, let
\[
 \delta(x)=\mathcal B(x,u_\nu(x))-\theta.
\]
The exact formula also gives the lower bound
\begin{equation}\label{eq:R-two-sided}
 \cot\delta(x)-\Xi(P,Z)
 \leq u_{\nu\nu}(x)
 \leq\Xi(P,Z)+
 \bigl[w_*(1+P^2)+\Xi(P,Z)\bigr]\cot\delta(x).
\end{equation}
Indeed, $\alpha_1\geq1$, $|\beta_1|\leq\Xi$, and $\cot\delta\geq0$.

\begin{proof}[Proof of Corollary \ref{cor:boundary-hessian-intro}]
The boundary identity \eqref{eq:tangential-block-intro}, the bounds for
$u_\nu$ and $\varphi$, and the fixed boundary geometry control the tangential
block $M$.  Theorem \ref{thm:mixed-intro} controls the mixed block $z$ by a
constant $Z$ independent of $t$.  Proposition \ref{prop:double-normal} then
gives
\[
 u_{\nu\nu}\leq C_0+C_1\cot\sigma.
\]
Combining the estimates for the tangential, mixed, and double-normal blocks
yields the asserted boundary Hessian bound.  The displayed dependencies in
Theorem \ref{thm:mixed-intro} and Lemma \ref{lem:alpha-beta} show that
$C_0$ and $C_1$ are independent of both $t$ and $\sigma$.
\end{proof}

The dependence on $Z$ in Proposition \ref{prop:double-normal} is essential.  A fixed positive limiting gap does not control the double-normal entry if the mixed block is allowed to grow.

\begin{example}\label{ex:mixed-jet}
In dimension two, at $Du=0$, consider
\begin{equation*}
 X_R=
 \begin{pmatrix}
  1&\sqrt{R-1}\\
  \sqrt{R-1}&R
 \end{pmatrix},
 \qquad R>1.
\end{equation*}
Then $X_R>0$ and $\det X_R=1$.  If $\lambda_1,\lambda_2$ are its eigenvalues, the positive-branch identity for the tangent of a sum gives
\[
 \arctan\lambda_1+\arctan\lambda_2=\frac\pi2.
\]
On the other hand, keeping the tangential entry $M=1$ fixed and sending the normal entry to infinity gives the limiting phase
\[
 L=\frac\pi2+\arctan1=\frac{3\pi}{4}.
\]
Thus the gap is the fixed number $\pi/4$, while
\[
 R\to\infty,
 \qquad |z|=\sqrt{R-1}\to\infty.
\]
This is why the mixed estimate and the normal limiting-phase condition are logically separate parts of the boundary argument.
\end{example}

\section{From the boundary to global curvature control}\label{sec:global}

For convex constant-phase graphs, the boundary Hessian estimate propagates globally by a direct maximum principle for the mean curvature.  We include the calculation because it is short and keeps the compactness theorem independent of any interior estimate from the literature.

Let $h_{ij}$ be the second fundamental form of $\Gamma_u$, let
\[
 H=\sum_i\kappa_i
\]
be its mean curvature, and let $\nabla$ denote the induced covariant derivative.  We use an orthonormal frame on the graph.

The mean curvature satisfies the following boundary maximum principle.

\begin{proposition}\label{prop:H-max}
Let $\Gamma_u$ be a smooth convex graph satisfying
\begin{equation}\label{eq:curvature-phase-global}
 F(h)=\sum_i\arctan\kappa_i=\theta.
\end{equation}
Then
\begin{equation}\label{eq:H-max}
 \sup_{\Gamma_u}H\leq\sup_{\partial\Gamma_u}H.
\end{equation}
\end{proposition}

\begin{proof}
Differentiating \eqref{eq:curvature-phase-global} once and twice gives
\begin{equation}\label{eq:F-first-second}
 F^{ij}\nabla_kh_{ij}=0,
 \qquad
 F^{ij}\Delta h_{ij}
 =-F^{ij,rs}\nabla_kh_{ij}\nabla_kh_{rs}.
\end{equation}
The Euclidean Simons identity is
\begin{equation}\label{eq:Simons}
 \nabla_i\nabla_jH
 =\Delta h_{ij}-Hh_{im}h_{mj}+\tr(h^2)h_{ij}.
\end{equation}
Contracting \eqref{eq:F-first-second} and \eqref{eq:Simons}, we obtain
\begin{align}
 F^{ij}\nabla_i\nabla_jH
 ={}&-F^{ij,rs}\nabla_kh_{ij}\nabla_kh_{rs}\notag\\
  &+F^{ij}h_{ij}\tr(h^2)
  -HF^{ij}h_{im}h_{mj}.
 \label{eq:LH}
\end{align}
The first term is nonnegative because $F$ is concave on the positive cone.
To determine the remaining algebraic term, choose an orthonormal frame in
which $h$ is diagonal and put
\[
 f_i=\frac1{1+\kappa_i^2}.
\]
Then $F^{ij}=f_i\delta_{ij}$, and the last two terms in \eqref{eq:LH}
equal
\begin{equation*}
 \begin{aligned}
 &\left(\sum_i f_i\kappa_i\right)
   \left(\sum_j\kappa_j^2\right)
 -\left(\sum_j\kappa_j\right)
   \left(\sum_i f_i\kappa_i^2\right)\\
 &\qquad
 =\sum_{i,j}f_i\kappa_i\kappa_j(\kappa_j-\kappa_i)\\
 &\qquad
 =\frac12\sum_{i,j}
 (f_j-f_i)(\kappa_i-\kappa_j)\kappa_i\kappa_j.
 \end{aligned}
\end{equation*}
Since
\begin{equation*}
 (f_j-f_i)(\kappa_i-\kappa_j)
 =\frac{(\kappa_i+\kappa_j)(\kappa_i-\kappa_j)^2}
 {(1+\kappa_i^2)(1+\kappa_j^2)}\geq0,
\end{equation*}
we have $F^{ij}\nabla_i\nabla_jH\geq0$.  Since $F^{ij}$ is positive
definite at every finite curvature matrix, the maximum principle proves
\eqref{eq:H-max}.  The calculation already allows zero principal curvatures:
all displayed expressions are continuous for $\kappa_i\geq0$, so the
semidefinite case follows directly by continuity and requires no perturbed
solution.
\end{proof}

The same argument controls the vertical homotopy uniformly.

\begin{proposition}\label{prop:homotopy-global}
Let $u$ be a convex solution of $G_t(Du,D^2u)=\theta$.  If $|Du|\leq P$, then
\begin{equation}\label{eq:homotopy-global}
 \sup_{\overline\Om}|D^2u|
 \leq C(n,P)\left(1+\sup_{\partial\Om}|D^2u|\right)
\end{equation}
uniformly for $t\in[0,1]$.
\end{proposition}

\begin{proof}
For $t>0$, consider the graph of $U=tu$.  Its principal curvatures are
\[
 \kappa_i^t=t\lambda_i(A_t[u]),
\]
and denote its mean curvature by $H_t=\sum_i\kappa_i^t$.  Its equation is
\begin{equation*}
 f_t(\kappa^t)=\sum_i\arctan(\kappa_i^t/t)=\theta.
\end{equation*}
 The function $f_t$ is elliptic and concave for $\kappa^t\geq0$.  In fact,
 \[
  (f_t)_i=\frac{t}{t^2+(\kappa_i^t)^2},
 \]
 and the algebraic term in the mean-curvature calculation has the required sign because
 \[
  \bigl((f_t)_j-(f_t)_i\bigr)(\kappa_i^t-\kappa_j^t)
  =\frac{t(\kappa_i^t+\kappa_j^t)(\kappa_i^t-\kappa_j^t)^2}
  {[t^2+(\kappa_i^t)^2][t^2+(\kappa_j^t)^2]}
  \geq0.
 \]
 Repeating Proposition \ref{prop:H-max} therefore gives
\[
 \sup_{\Gamma_{tu}}H_t\leq\sup_{\partial\Gamma_{tu}}H_t.
\]
 Moreover,
 \[
  H_t=t\tr A_t[u]
  =\frac{t}{w_t}\tr\bigl(B_t^2D^2u\bigr),
 \]
 and, since $D^2u\geq0$ and $w_t^{-2}I\leq B_t^2\leq I$,
 \[
  \frac{t}{w_t^3}\tr D^2u
  \leq H_t
  \leq\frac{t}{w_t}\tr D^2u.
 \]
 Together with $|D^2u|\leq\tr D^2u\leq n|D^2u|$ and $w_t\leq(1+P^2)^{1/2}$, the boundary maximum principle for $H_t$, followed by division by $t$, proves \eqref{eq:homotopy-global} uniformly for $t>0$.

When $t=0$, the equation is $F(D^2u)=\theta$.  For a fixed unit vector $\xi$, twice differentiating in the $\xi$ direction gives
\[
 F^{ij}(u_{\xi\xi})_{ij}
 =-F^{ij,rs}u_{ij\xi}u_{rs\xi}\geq0.
\]
The maximum principle bounds $u_{\xi\xi}$ by its boundary maximum.  Taking the supremum over $\xi$ proves the result at $t=0$.
\end{proof}

We now prove Theorem \ref{thm:compactness-intro}.  First observe that the actual gap is always positive for a smooth finite-curvature solution.  Indeed, Lemma \ref{lem:exact-Schur} gives
\begin{equation*}
 \delta(x)=\arctan\frac{\alpha_1(x)}
 {u_{\nu\nu}(x)-\beta_1(x)}>0.
\end{equation*}
The convention at a zero denominator gives $\delta=\pi/2$.  Continuity and compactness of the boundary imply $\delta_*(u)>0$.

\begin{proof}[Proof of Theorem \ref{thm:compactness-intro}]
On the boundary,
\begin{equation*}
 M=D^2_{\partial\Om}\varphi-u_\nu\II_{\partial\Om}.
\end{equation*}
Thus the first two bounds in \eqref{eq:compact-data-intro} control the
tangential block of $D^2u$, while the third controls the mixed block.  Under
the hypotheses of Theorem \ref{thm:mixed-intro}, this third bound is supplied
by the mixed boundary estimate; in the present theorem it is included in the
definition of controlled data.  At every boundary point,
\eqref{eq:R-two-sided} gives
\begin{equation*}
 u_{\nu\nu}\leq C(1+\cot\delta_*(u)).
\end{equation*}
Therefore
\begin{equation}\label{eq:boundary-Hessian-delta}
 \sup_{\partial\Om}|D^2u|
 \leq C(1+\cot\delta_*(u)).
\end{equation}

The graph matrix relation
\begin{equation*}
 D^2u=wB(Du)^{-1}A[u]B(Du)^{-1}
\end{equation*}
and $|Du|\leq P$ give
\begin{equation}\label{eq:norm-comparison}
 |A[u]|\leq|D^2u|
 \leq(1+P^2)^{3/2}|A[u]|.
\end{equation}
Because all principal curvatures are nonnegative, $|A[u]|\leq H$.  Proposition \ref{prop:H-max}, \eqref{eq:boundary-Hessian-delta}, and \eqref{eq:norm-comparison} prove \eqref{eq:global-upper-intro}.

Choose $x_*\in\partial\Om$ where the minimum in \eqref{eq:delta-star-intro} is attained.  The lower half of \eqref{eq:R-two-sided} gives
\begin{equation*}
 u_{\nu\nu}(x_*)
 \geq\cot\delta_*(u)-C,
\end{equation*}
which proves \eqref{eq:global-lower-intro}.  The equivalences in \eqref{eq:compact-equivalence-intro} follow immediately from the upper and lower bounds and the norm comparison.
\end{proof}

Combining the mixed estimate, a normal-window gap, and Proposition \ref{prop:homotopy-global} gives the following a priori estimate.

\begin{corollary}\label{cor:homotopy-C2}
Assume the hypotheses of Theorem \ref{thm:mixed-intro}.  Suppose that every boundary normal derivative lies in a fixed compact window $J_x$ and that
\begin{equation}\label{eq:window-gap-global}
 \inf_{\substack{x\in\partial\Om,\,s\in J_x\\ M(x,s)\geq0}}
 \bigl(\mathcal B(x,s)-\theta\bigr)\geq\sigma>0.
\end{equation}
Then every convex solution of the homotopy satisfies
\begin{equation*}
 \|u\|_{C^2(\overline\Om)}
 \leq C_0+C_1\cot\sigma,
\end{equation*}
where the constants are independent of $t\in[0,1]$ and $\sigma$.
\end{corollary}

\begin{remark}\label{rem:two-recession-conditions}
(MRC) is used only to control $u_{\xi\nu}$ for tangential unit vectors
$\xi$.  The limiting-phase gap \eqref{eq:window-gap-global} is used only
after that estimate, to control $u_{\nu\nu}$.  Example \ref{ex:mixed-jet} shows that the normal gap cannot replace (MRC), while the
formula \eqref{eq:R-alpha-beta} shows that (MRC) cannot prevent double-normal
blow-up if the actual gap collapses.
\end{remark}

\section{Sharpness models}\label{sec:radial}

This section gives two sharpness examples.  A radial family attains the optimal boundary-gap blow-up rate while all lower-order boundary data remain controlled.  A rank-loss model then shows that a strict lower subsolution does not by itself preserve strict convexity.

We now prove Theorem \ref{thm:radial-intro} by constructing a family of
solutions on one fixed ball.

\begin{proof}[Proof of Theorem \ref{thm:radial-intro}]
\medskip
\noindent\emph{Step 1: the radial equation.}

Let $u(x)=f(r)$, where $r=|x|$, and put
\begin{equation*}
 q(r)=f'(r),
 \qquad
 s(r)=\frac{q(r)}{\sqrt{1+q(r)^2}},
 \qquad
 \zeta(r)=\frac{s(r)}r.
\end{equation*}
The radial and tangential principal curvatures are
\begin{equation}\label{eq:radial-curvatures}
 \kappa_r=s'(r)=\zeta+r\zeta',
 \qquad
 \kappa_T=\frac{s(r)}r=\zeta,
\end{equation}
where $\kappa_T$ has multiplicity $n-1$.  Thus the equation becomes
\begin{equation*}
 \arctan(\zeta+r\zeta')+(n-1)\arctan\zeta=\theta.
\end{equation*}

Assume
\[
 \frac\pi2\leq\theta<\frac{n\pi}{2}
\]
and define the endpoint angle and the two corresponding slopes by
\begin{equation*}
 \vartheta_*=\frac{\theta-\pi/2}{n-1},
 \qquad \zeta_*=\tan\vartheta_*,
 \qquad \zeta_0=\tan(\theta/n).
\end{equation*}
Since
\begin{equation*}
 \frac\theta n-\vartheta_*
 =\frac{n\pi/2-\theta}{n(n-1)}>0,
\end{equation*}
we have $0\leq\zeta_*<\zeta_0$.

Fix $0<R_0<R_1$ so small that
\begin{equation}\label{eq:R1-z0}
 R_1\zeta_0<1.
\end{equation}
For
\begin{equation*}
 0<\delta<\delta_0<\frac\pi2-\frac\theta n,
\end{equation*}
let $\zeta_\delta$ solve
\begin{equation}\label{eq:z-ODE}
 r\zeta_\delta'
 =\tan\bigl(\theta-(n-1)\arctan\zeta_\delta\bigr)-\zeta_\delta,
\end{equation}
with initial value
\begin{equation*}
 \zeta_\delta(R_0)
 =\tan\left(\vartheta_*+\frac\delta{n-1}\right).
\end{equation*}
The right hand side of \eqref{eq:z-ODE}, viewed as a function of $\zeta$,
is positive on $(\zeta_*,\zeta_0)$ and vanishes at $\zeta_0$.
Uniqueness for the autonomous equation in
logarithmic time gives
\begin{equation*}
 \zeta_\delta(R_0)<\zeta_\delta(r)<\zeta_0,
 \qquad \zeta_\delta'(r)>0
 \quad (R_0\leq r\leq R_1).
\end{equation*}

Set
\begin{equation*}
 f_\delta'(r)=
 \frac{r\zeta_\delta(r)}{\sqrt{1-r^2\zeta_\delta(r)^2}},
 \qquad
 f_\delta(R_0)=0.
\end{equation*}
Condition \eqref{eq:R1-z0} makes the first derivatives uniformly bounded.  Equations \eqref{eq:z-ODE} and \eqref{eq:radial-curvatures} give
\begin{equation*}
 \kappa_r
 =\tan\bigl(\theta-(n-1)\arctan\zeta_\delta\bigr),
 \qquad \kappa_T=\zeta_\delta.
\end{equation*}
Hence $u_\delta(x)=f_\delta(|x|)$ is a smooth strictly convex solution of
$F(A[u_\delta])=\theta$ on $R_0<|x|<R_1$.  At $r=R_0$,
\begin{equation*}
 \kappa_r(R_0)=\cot\delta.
\end{equation*}
Moreover, since
\begin{equation*}
 f_\delta''
 =\frac{\kappa_r}{(1-r^2\zeta_\delta^2)^{3/2}},
\end{equation*}
we have
\begin{equation}\label{eq:f-second-blow}
 f_\delta''(R_0)\asymp\cot\delta\asymp\delta^{-1}.
\end{equation}

\medskip
\noindent\emph{Step 2: the boundary gap under a rank-one escape.}

We shall restrict the radial solutions to a fixed ball tangent to the sphere $\{|x|=R_0\}$.  The following rank-one formula identifies the limiting gap at an arbitrary boundary point.

\begin{lemma}\label{lem:rank-one}
Let $A>0$ and $v\neq0$.  Put
\begin{equation*}
 \mathcal N_v=\ip{(I+A^2)^{-1}v}{v},
 \qquad
 \mathcal D_v=\ip{A(I+A^2)^{-1}v}{v}.
\end{equation*}
Then
\begin{equation}\label{eq:rank-one-limit}
 \lim_{T\to\infty}
 \bigl(F(A+Tv\otimes v)-F(A)\bigr)
 =\arctan\frac{\mathcal N_v}{\mathcal D_v}.
\end{equation}
If $\kappa_{\min}$ and $\kappa_{\max}$ are the extreme eigenvalues of $A$, then
\begin{equation*}
 \frac1{\kappa_{\max}}
 \leq\frac{\mathcal N_v}{\mathcal D_v}
 \leq\frac1{\kappa_{\min}}.
\end{equation*}
\end{lemma}

\begin{proof}
The matrix determinant lemma gives
\begin{equation}\label{eq:rank-one-det}
 \det(I+\sqrt{-1}(A+Tv\otimes v))
 =\det(I+\sqrt{-1} A)
 \left(1+\sqrt{-1} T v^T(I+\sqrt{-1} A)^{-1}v\right).
\end{equation}
Since
\[
 (I+\sqrt{-1} A)^{-1}=(I-\sqrt{-1} A)(I+A^2)^{-1},
\]
 take the continuous lift of the determinant phase along $T\geq0$, starting at $T=0$.  The last factor in \eqref{eq:rank-one-det} lies in the open first quadrant and has limiting argument
 $\arctan(\mathcal N_v/\mathcal D_v)$.  This proves \eqref{eq:rank-one-limit}.  In an eigenbasis of $A$,
\begin{equation*}
 \frac{\mathcal D_v}{\mathcal N_v}
 =\frac{\displaystyle\sum_i
 \frac{v_i^2\kappa_i}{1+\kappa_i^2}}
 {\displaystyle\sum_i\frac{v_i^2}{1+\kappa_i^2}},
\end{equation*}
which is a weighted average of the $\kappa_i$.  The bounds follow.
\end{proof}

Increasing the Euclidean double-normal Hessian entry changes the curvature matrix by a positive rank-one term.  Thus Lemma \ref{lem:rank-one} computes exactly $\mathcal B-F(A)$ for any chosen boundary normal.

\medskip
\noindent\emph{Step 3: restriction to a fixed ball.}

Choose $\rho>0$ so that
\begin{equation*}
 R_0+2\rho<R_1,
\end{equation*}
and set
\begin{equation*}
 \Om=B_\rho\bigl((R_0+\rho)e_1\bigr).
\end{equation*}
The geometry is summarized in Figure \ref{fig:fixed-ball}.
\begin{figure}[H]
\centering
\begin{tikzpicture}[scale=0.88,>=stealth]
 \coordinate (O) at (0,0);
 \coordinate (X) at (3,0);
 \coordinate (C) at (4.4,0);
 \draw[->] (-0.35,0)--(6.15,0) node[right] {$e_1$};
 \draw[gray,dashed] (X) arc[start angle=0,end angle=35,radius=3];
 \draw[gray,dashed] (X) arc[start angle=0,end angle=-35,radius=3];
 \node[gray] at (2.25,1.55) {$|x|=R_0$};
 \draw[thick] (C) circle (1.4);
 \draw[<->] (C)--(4.4,1.4) node[midway,right] {$\rho$};
 \node at (5.15,-0.55) {$\Omega$};
 \fill (O) circle (1.5pt) node[below=3pt] {$0$};
 \fill (X) circle (1.7pt) node[below left=3pt] {$x_0=R_0e_1$};
 \fill (C) circle (1.5pt);
\end{tikzpicture}
\caption{The fixed ball tangent to the sphere $|x|=R_0$ at $x_0$.}
\label{fig:fixed-ball}
\end{figure}
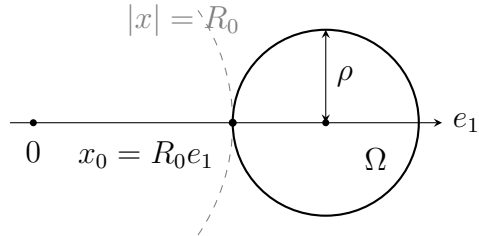
Then $\overline\Om\subset\{R_0\leq|x|\leq R_1\}$, and
\begin{equation*}
 x_0=R_0e_1
\end{equation*}
is the unique point of $\overline\Om$ with radius $R_0$.  The inward unit normal of $\partial\Om$ at $x_0$ is $e_1$, the radial direction.

The radial curvature $\kappa_r$ is decreasing in $r$, because
$\zeta_\delta$ is increasing and
\begin{equation*}
 \frac{d}{d\zeta}
 \tan\bigl(\theta-(n-1)\arctan\zeta\bigr)
 =-(n-1)\frac{1+\kappa_r^2}{1+\zeta^2}<0.
\end{equation*}
By Step 1, $\kappa_T=\zeta_\delta<\zeta_0$.  After reducing
$\delta_0$ if necessary,
\[
 \zeta_0<\cot\delta.
\]
Moreover, the monotonicity of $\kappa_r$ gives
$\kappa_r(r)\leq\kappa_r(R_0)=\cot\delta$.  Combining the radial and
tangential estimates yields
\begin{equation}\label{eq:kappa-max-radial}
 \max_i\kappa_i(x)\leq\cot\delta
 \qquad (x\in\overline\Om),
\end{equation}
with equality at $x_0$ in the radial direction.

At any boundary point, Lemma \ref{lem:rank-one} and
\eqref{eq:kappa-max-radial} give
\[
 \frac{\mathcal N_v}{\mathcal D_v}
 \geq\frac1{\kappa_{\max}}\geq\tan\delta.
\]
Hence
\begin{equation}\label{eq:gap-lower-radial}
 \mathcal B(x,u_{\delta,\nu}(x))-\theta
 \geq\arctan(\tan\delta)=\delta.
\end{equation}
At $x_0$, the escaping rank-one direction is the radial principal direction and $\kappa_r(x_0)=\cot\delta$.  Equality holds in \eqref{eq:gap-lower-radial}.  Hence
\begin{equation}\label{eq:gap-equality-radial}
 \delta_*(u_\delta)=\delta.
\end{equation}
Since we use the operator norm for symmetric matrices, \eqref{eq:kappa-max-radial} also gives
\begin{equation}\label{eq:A-equality-radial}
 \sup_{\overline\Om}|A[u_\delta]|=\cot\delta.
\end{equation}

\medskip
\noindent\emph{Step 4: uniform boundary data and mixed derivatives.}

It remains to verify that the quantities other than the double-normal derivative stay uniformly bounded.  Define
\begin{equation*}
 \eta(r)=\frac\pi2+(n-1)\arctan\zeta_\delta(r)-\theta.
\end{equation*}
Then $\eta(R_0)=\delta$, $\kappa_r=\cot\eta$, and the ODE gives
\begin{equation}\label{eq:eta-ODE}
 r\eta'
 =\frac{n-1}{1+\zeta_\delta^2}(\cot\eta-\zeta_\delta).
\end{equation}
 Choose $\bar\eta>0$ so small that, whenever $0<\eta\leq\bar\eta$
 and $\zeta\in[\zeta_*,\zeta_0]$, the right hand side of
 \eqref{eq:eta-ODE} is bounded above and below by fixed positive multiples
 of $1/\eta$.  Take $\delta_0<\bar\eta/2$ and then decrease $R_1-R_0$ so that
 \[
  \delta_0^2+2C(R_1-R_0)<\bar\eta^2,
 \]
 where $C$ is the upper comparison constant.  A first-exit argument, applied to $(\eta^2)'=2\eta\eta'$, shows that $\eta$ cannot reach $\bar\eta$ on $[R_0,R_1]$.  Consequently all remaining factors stay in fixed compact intervals and
\begin{equation*}
 \frac c\eta\leq\eta'\leq\frac C\eta.
\end{equation*}
Integration yields
\begin{equation*}
 c(\delta^2+r-R_0)
 \leq\eta(r)^2
 \leq C(\delta^2+r-R_0).
\end{equation*}
 Differentiating \eqref{eq:eta-ODE} once and using
 $\eta'\asymp\eta^{-1}$ gives $|\eta''|\leq C\eta^{-3}$.
 The $\zeta$-equation also gives
 $|\zeta_\delta'|\leq C\eta^{-1}$ and
 $|\zeta_\delta''|\leq C\eta^{-3}$.
 To make the derivative calculation explicit, put
 \[
  \Pi(r,\zeta)=(1-r^2\zeta^2)^{-3/2}.
 \]
 Condition \eqref{eq:R1-z0} keeps $\Pi$ and its first two derivatives
 uniformly bounded on the relevant compact set.  Since
 $f_\delta''=\Pi(r,\zeta_\delta)\cot\eta$, differentiation gives
 \[
  f_\delta'''=
  (\Pi_r+\Pi_\zeta\zeta_\delta')\cot\eta
  -\Pi\csc^2\eta\,\eta'.
 \]
 Using $\cot\eta\leq C\eta^{-1}$,
 $\csc^2\eta\leq C\eta^{-2}$, and the preceding estimates yields
 $|f_\delta'''|\leq C\eta^{-3}$.  Differentiating once more produces only
 terms containing bounded derivatives of $\Pi$ multiplied by
 $\zeta_\delta'$, $\zeta_\delta''$, $\eta'$, and $\eta''$; the worst terms
 have order $\eta^{-5}$.  Together with
 $\eta^2\asymp\delta^2+r-R_0$, this proves
\begin{align}
 |f_\delta''(r)|
 &\leq\frac C{\sqrt{\delta^2+r-R_0}},
 & |f_\delta'''(r)|
 &\leq\frac C{(\delta^2+r-R_0)^{3/2}},\notag\\
 |f_\delta^{(4)}(r)|
 &\leq\frac C{(\delta^2+r-R_0)^{5/2}}.
 \label{eq:f-derivative-radial-bounds}
\end{align}

Use local tangential coordinates $y\in\R^{n-1}$ on $\partial\Om$ near $x_0$.  The radius $r=r(y)$ satisfies
\begin{equation}\label{eq:r-y-geometry}
 r(y)-R_0\asymp|y|^2,
 \qquad |D_yr(y)|\leq C|y|.
\end{equation}
For completeness, the first three tangential derivatives of the boundary
trace are
\[
 \varphi_\delta(y)=f_\delta(r(y))
\]
and
\begin{align*}
 D\varphi_\delta
 &=f_\delta'Dr,\\
 D^2\varphi_\delta
 &=f_\delta''Dr\otimes Dr+f_\delta'D^2r,\\
D^3\varphi_\delta
&=f_\delta'''Dr^{\otimes3}
 +3f_\delta''\operatorname{sym}(D^2r\otimes Dr)
 +f_\delta'D^3r.
\end{align*}
Here $\operatorname{sym}$ denotes full symmetrization of the tensor product.
The derivatives of $r(y)$ of order at least two are uniformly bounded.
Using \eqref{eq:f-derivative-radial-bounds}, the only potentially singular
terms are bounded by
\begin{equation*}
 |f_\delta''|\,|D_yr|
 \leq C\frac{|y|}{\sqrt{\delta^2+|y|^2}}\leq C,
\end{equation*}
and
\begin{equation*}
 |f_\delta'''|\,|D_yr|^3
 \leq C\frac{|y|^3}{(\delta^2+|y|^2)^{3/2}}\leq C.
\end{equation*}
All other terms are easier, and points away from $x_0$ are uniformly regular.  Thus
\begin{equation*}
 \sup_{0<\delta<\delta_0}
 \|u_\delta|_{\partial\Om}\|_{C^3(\partial\Om)}<\infty.
\end{equation*}

Let $e_r=x/|x|$ be the radial unit vector.  The Euclidean radial Hessian is
\begin{equation*}
 D^2u_\delta
 =f_\delta''e_r\otimes e_r
 +\frac{f_\delta'}r(I-e_r\otimes e_r).
\end{equation*}
For a unit tangent vector $\xi$ and the inward normal $\nu$ of the fixed ball,
\begin{equation*}
 D^2u_\delta(\xi,\nu)
 =\left(f_\delta''-\frac{f_\delta'}r\right)
 (\xi\cdot e_r)(\nu\cdot e_r).
\end{equation*}
Near $x_0$, $|\xi\cdot e_r|\leq C|y|$.  Equations \eqref{eq:f-derivative-radial-bounds} and \eqref{eq:r-y-geometry} imply
\begin{equation*}
 \sup_{0<\delta<\delta_0}
 \sup_{x\in\partial\Om}
 \sup_{\substack{\xi\in T_x\partial\Om\\ |\xi|=1}}
 |D^2u_\delta(\xi,\nu)|<\infty.
\end{equation*}
The $C^1$ bound follows directly from
$r\zeta_\delta(r)\leq R_1\zeta_0<1$.

\medskip
\noindent\emph{Step 5: the sharp $C^4$ growth.}
Restrict to any tangential coordinate axis through $y=0$.  In the ball
geometry one has
\[
 r_y(0)=0,
 \qquad r_{yy}(0)=\frac1\rho+\frac1{R_0}>0.
\]
The fourth derivative satisfies
\begin{equation*}
 \partial_y^4\varphi_\delta(0)
 =3f_\delta''(R_0)r_{yy}(0)^2
   +f_\delta'(R_0)r_{yyyy}(0),
\end{equation*}
 and hence is bounded below by a positive multiple of $\delta^{-1}$.  Conversely, the fourth-order chain rule and \eqref{eq:r-y-geometry} reduce all possibly singular terms to
 \[
  |f_\delta^{(4)}|\,|D_yr|^4,
  \quad |f_\delta'''|\,|D_yr|^2,
  \quad |f_\delta''|\bigl(1+|D_yr|\bigr),
 \]
 each of which is bounded by $C\delta^{-1}$ using
 \eqref{eq:f-derivative-radial-bounds}.  Thus the boundary $C^4$ norm has
 exact order $\delta^{-1}$.  Equations \eqref{eq:f-second-blow},
 \eqref{eq:gap-equality-radial}, and \eqref{eq:A-equality-radial}, together
 with the uniform $C^1$, boundary $C^3$, and mixed derivative bounds proved
 above, establish all assertions of Theorem \ref{thm:radial-intro}.
\end{proof}

\begin{remark}\label{rem:annulus}
The same functions on the fixed annulus $\{R_0<|x|<R_1\}$ have constant boundary traces on each component, and therefore uniformly bounded boundary data in every $C^m$ norm.  The curvature and gap still satisfy $\max|A|=\cot\delta$ and $\min(\mathcal B-\theta)=\delta$ at the inner boundary.  This gives a simple geometric sharpness model, although the annulus is not a convex domain.
\end{remark}

\begin{remark}\label{rem:low-phase}
If $0<\theta<\pi/2$ and $u$ is convex, then $M\geq0$ on the boundary and
\[
 \mathcal B(x,u_\nu)\geq\frac\pi2.
\]
Therefore
\[
 \delta_*(u)\geq\frac\pi2-\theta>0.
\]
The convex double-normal gap cannot collapse in this range.  The high-phase radial family above begins exactly at $\theta=\pi/2$.
\end{remark}

The closed convex branch is essential here: a strict convex lower subsolution does not force every solution with the same boundary data to be strictly convex.  The following model gives an explicit obstruction.

\begin{equation}\label{eq:cylinder-problem}
 F(A[u])=\theta\quad\text{in }\Om,
 \qquad u=\varphi\quad\text{on }\partial\Om.
\end{equation}
For this problem, a \emph{strict lower subsolution} is a strictly convex function $\underline u$ satisfying $\underline u=\varphi$ on $\partial\Om$ and $F(A[\underline u])>\theta$ on $\overline\Om$.

\begin{proposition}\label{prop:cylinder}
There are smooth data $(\Om,\varphi,\theta)$, with $\Om\subset\R^2$ uniformly convex and $0<\theta<\pi/2$, for which a strict lower subsolution exists but the unique smooth convex solution of \eqref{eq:cylinder-problem} is not strictly convex.
\end{proposition}

\begin{proof}
Let $\Om=B_1(0)\subset\R^2$, choose $R>2$, and define
\begin{equation*}
 u_0(x_1,x_2)=-\sqrt{R^2-x_1^2}.
\end{equation*}
Put $s=(R^2-x_1^2)^{1/2}$.  Then
\[
 Du_0=(x_1/s,0),
 \qquad
 D^2u_0=\diag(R^2/s^3,0).
\]
The graph is a piece of a circular cylinder.  Its principal curvatures are
\begin{equation*}
 \kappa_1=\frac1R,
 \qquad \kappa_2=0.
\end{equation*}
Thus $u_0$ is a smooth convex solution of \eqref{eq:cylinder-problem} with
\begin{equation*}
 \theta=\arctan(1/R),
 \qquad \varphi=u_0|_{\partial B_1}.
\end{equation*}

For $\eps>0$, set
\begin{equation*}
 \underline u_\eps=u_0+\eps(|x|^2-1).
\end{equation*}
 It has the same boundary value, satisfies $\underline u_\eps<u_0$ in $B_1$, and has $D^2\underline u_\eps>0$.  A direct differentiation of the graph phase at $\eps=0$ gives
\begin{equation*}
 \left.\frac{d}{d\eps}
 F(A[\underline u_\eps])\right|_{\eps=0}
 =\frac{2s}{R(R^2+1)}
 \bigl(2R^2+1-4x_1^2\bigr)>0
 \qquad\text{on }\overline B_1.
\end{equation*}
The positivity is uniform.  Consequently, for every sufficiently small
$\eps>0$, the function $\underline u_\eps$ is smooth and strictly convex,
and it is a strict lower subsolution because
\[
 F(A[\underline u_\eps])>\theta
 \quad\text{on }\overline B_1.
\]

It remains to note uniqueness on the convex branch.  If $v$ is another smooth convex solution with the same boundary value, integrate the full linearization of $G_1$ along
$u_s=(1-s)u_0+sv$.  With
\[
 \bar a^{ij}=\int_0^1G_{1,X_{ij}}(Du_s,D^2u_s)\,ds,
 \qquad
 \bar b^k=\int_0^1G_{1,p_k}(Du_s,D^2u_s)\,ds,
\]
the difference $v-u_0$ satisfies the linear equation
\[
 \bar a^{ij}(v-u_0)_{ij}+\bar b^k(v-u_0)_k=0
\]
with a uniformly positive definite coefficient matrix $\bar a$; smoothness and compactness supply uniform ellipticity.  The maximum principle and the zero boundary value give $v=u_0$.  Since $D^2u_0$ has rank one, the unique convex solution is not strictly convex.
\end{proof}

\begin{remark}\label{rem:rank-loss-consequence}
Proposition \ref{prop:cylinder} is independent of the boundary recession estimates.  It shows that an existence theorem which insists on strict convexity needs a separate rank-preservation or quantitative lower-curvature argument.  A uniform upper Hessian estimate, even together with a strict lower subsolution, does not provide such a lower bound.
\end{remark}

\appendix

\section{Four model situations satisfying (MRC)}
\label{app:MRC-examples}

We record four simple mechanisms which imply the mixed recession
condition.  They are intended to illustrate the meaning of
(MRC), rather than to provide the most general verification
criteria.

\begin{example}
	\label{ex:MRC-Hessian-endpoint}
	Consider the fixed operator
	\[
	G_0(p,X)=F(X)=\sum_{i=1}^n\arctan\lambda_i(X)
	\]
	on the unit ball $\Om=B_1(0)$, with zero boundary value.  Let
	\[
	\underline u(x)=\frac A2\bigl(|x|^2-1\bigr),
	\qquad
	A>\tan\frac{\theta}{n}.
	\]
	Then
	\[
	F(D^2\underline u)=n\arctan A>\theta.
	\]
	At $t=0$ one has $b=0$.  If $X\geq0$ and $F(X)=\theta$, the concavity
	of $F$ gives
	\[
	\cJ
	=F_X(X):(AI-X)
	\geq F(AI)-F(X)
	=n\arctan A-\theta
	=:\eps_0>0.
	\]
	Moreover, $\II_{\partial B_1}=I_T$, and hence
	\[
	g=-\ell=a_T:\II_{\partial B_1}=\tr a_T\geq0.
	\]
	Thus every recession pair satisfies
	\[
	j+\tau g\geq\eps_0
	\]
	for every $\tau>0$.  Hence the fixed-operator form of (MRC)
	holds.  This is the usual strict-subsolution mechanism without any
	gradient drift.
\end{example}

\begin{example}
	\label{ex:MRC-boundary-curvature}
	Suppose that, on the relevant level jets,
	\[
	\tr a\geq c_0>0,
	\qquad
	|b|\leq C_b,
	\]
	and assume that near $x_0$,
	\[
	\II_{\partial\Om}\geq\kappa_0 I_T,
	\qquad
	\kappa_0c_0>C_b.
	\]
	Along a doubly degenerate recession sequence,
	$a^{(m)}(Dd,Dd)\to0$.  Since $Dd(x_m)\to\nu(x_0)$ and $a^{(m)}$ is uniformly
	bounded,
	\[
	\tr(a^{(m)})_{T_{x_0}}
	=\tr a^{(m)}-a^{(m)}(\nu(x_0),\nu(x_0))
	\geq c_0-o(1).
	\]
	Consequently,
	\[
	\begin{aligned}
		g^{(m)}
		&=(a^{(m)})_{T_{x_0}}:\II_{\partial\Om}(x_0)
		-b^{(m)}\cdot\nu(x_0)+o(1)\\
		&\geq \kappa_0c_0-C_b-o(1).
	\end{aligned}
	\]
	Hence every $(j,g)\in\cR_{x_0}^*$ satisfies
	$g\geq\gamma_0$ for some $\gamma_0>0$.  Since $j$ is uniformly
	bounded below, choosing $\tau>0$ sufficiently large proves
	(MRC).
	
	In particular, for $\Om=B_R(0)$ one has $\kappa_0=R^{-1}$.
	Thus, for a fixed compact gradient set, every sufficiently small ball
	satisfying
	\[
	\frac{c_0}{R}>C_b
	\]
	has the required drift-corrected boundary convexity.
\end{example}

\begin{example}
	\label{ex:MRC-full-jet}
	Assume that there is $\eta>0$ such that
	\[
	G_t\bigl(p,D^2\underline u(x)\bigr)
	\geq\theta+\eta
	\]
	for all $x$ near $x_0$, $t\in[0,1]$, and $p\in K$.  Put
	\[
	\begin{aligned}
	 C_b&=\sup\bigl\{|G_{t,p}(p,X)|:
	 (t,p,X)\text{ is a level jet with }p\in K\bigr\},\\
	 D_0&=\sup_{\substack{x\ \mathrm{near}\ x_0\\ p\in K}}
	 |D\underline u(x)-p|.
	\end{aligned}
	\]
	and suppose that
	\[
	\eta>C_bD_0.
	\]
	For every level jet, concavity in the Hessian variable gives
	\[
	\begin{aligned}
		\cJ
		&=
		a:(D^2\underline u-X)
		+b\cdot(D\underline u-p)\\
		&\geq
		G_t(p,D^2\underline u)-G_t(p,X)-C_bD_0\\
		&\geq\eta-C_bD_0
		=:\eps_0>0.
	\end{aligned}
	\]
	The quantity $g=-\ell$ is uniformly bounded on the recession set.
	Therefore, after choosing $\tau>0$ sufficiently small,
	\[
	j+\tau g\geq\frac{\eps_0}{2}
	\qquad
	\text{for every }(j,g)\in\cR_{x_0}^*.
	\]
	Thus (MRC) holds.  This example shows that a sufficiently strong
	phase subsolution absorbs the possible loss caused by the gradient
	drift.
\end{example}

\begin{example}
	\label{ex:MRC-empty}
	Suppose that there are $R<\infty$ and $c>0$ such that every level jet
	near $x_0$ with $|X|\geq R$ satisfies at least one of
	\[
	a(Dd,Dd)\geq c,
	\qquad
	Q\geq c.
	\]
	Then no sequence can satisfy simultaneously
	\[
 |X_m|\longrightarrow\infty,
 \qquad
 a^{(m)}(Dd,Dd)\longrightarrow0,
 \qquad
 Q^{(m)}\longrightarrow0.
	\]
	Hence
	\[
	\cR_{x_0}^*=\varnothing.
	\]
	By the convention that the infimum over the empty set is $+\infty$,
	condition (MRC) holds automatically.  In this situation one of
	the two standard quadratic terms remains coercive, so no additional
	recession compatibility is needed.
\end{example}

\section*{Declarations}

\noindent\textbf{Data availability:} Data availability is not applicable to this article as no new data were created or analyzed in this study.

\medskip

\noindent\textbf{Conflict of interest:} The authors declare that they have no conflicts of interests.

\end{document}